\pdfoutput=1
\documentclass[a4paper,oneside,12pt]{amsart}
\usepackage{typearea} 
\typearea{15} 

\usepackage[UKenglish]{babel}
\usepackage[T1]{fontenc}

\usepackage{amsmath}
\usepackage{amsthm}
\usepackage{amsfonts}
\usepackage{amssymb}
\usepackage[neverdecrease]{paralist}
\usepackage{enumitem}

\usepackage{dsfont}
\usepackage{ifthen}
\usepackage{twoopt}
\usepackage{mathtools}
\usepackage{mathrsfs}
\usepackage{etoolbox}
\usepackage{xparse}
\usepackage{graphicx}

\usepackage{microtype}
\usepackage[plainpages=false,colorlinks=false]{hyperref}

\theoremstyle{plain}
\newtheorem{theorem}{Theorem}[section]
\newtheorem{proposition}[theorem]{Proposition}
\newtheorem{corollary}[theorem]{Corollary}
\newtheorem{lemma}[theorem]{Lemma}
\theoremstyle{definition}
\newtheorem{example}[theorem]{Example}
\newtheorem{remark}[theorem]{Remark}
\newtheorem{definition}[theorem]{Definition}

\mathtoolsset{centercolon}

\newcommand{\mathdcl}[1]{{\ifstrequal{#1}{l}{l}{\textup{#1}}}}

\newcommand{\Borel}{\mathscr{B}}
\newcommand{\mbdy}{\partial^\mathdcl{m}}
\newcommand{\sing}{\mathdcl{s}}
\newcommand{\loc}{\mathdcl{loc}}
\newcommand{\nloc}{\mathdcl{inf}}
\newcommand{\cpt}{\mathdcl{c}}
\newcommand{\rect}{\mathdcl{r}}
\newcommand{\urect}{\mathdcl{u}}

\newcommand{\RR}{\mathbb{R}}
\newcommand{\NN}{\mathbb{N}}
\newcommand{\one}{\mathds{1}}
\newcommand{\cM}{\mathcal{M}}

\newcommand{\clos}[1]{\overline{#1}}
\newcommand{\setmid}{:}

\newcommand{\symdiff}{\mathbin{\vartriangle}}
\DeclareMathOperator{\linspan}{span}

\DeclareMathOperator{\supp}{supp}
\DeclareMathOperator{\graph}{gr}
\DeclareMathOperator{\diam}{diam}
\DeclareMathOperator{\dist}{dist}
\DeclareMathOperator{\inter}{int}
\DeclareMathOperator{\ext}{ext}

\providecommand{\wedgeq}{\mathrel{\widehat=}}

\DeclarePairedDelimiter\pdp{\lparen}{\rparen}
\makeatletter
\@ifpackageloaded{xparse}{
\DeclareDocumentCommand{\dx}{O{x}d()}{%
\,\mathrm{d}#1\IfValueTF{#2}{\pdp{#2}}{}}
}{
\newcommand{\dx}[1][x]{\,\mathrm{d}#1}
}
\makeatother

\newcommand{\eps}{\varepsilon}

\newcommand{\Cinfc}[1][\Omega]{\ensuremath{C^\infty_\cpt(#1)}}
\newcommand{\Cbar}[1][\Omega]{\ensuremath{C(\overline{#1})}}

\newcommand{\Ltwo}[1][\Omega]{\ensuremath{L^2(#1)}}
\newcommand{\Linf}[1][\Omega]{\ensuremath{L^\infty(#1)}}
\newcommand{\Loneloc}[1][\LonelocARG]{\ensuremath{L^1_\loc(#1)}}

\newcommand{\Lone}[1][\Omega]{\ensuremath{L^1(#1)}}
\newcommandtwoopt{\Lp}[2][p][\Omega]{\ensuremath{L^{#1}(#2)}}
\newcommandtwoopt{\Wt}[2][p][\Omega]{\ensuremath{\widetilde{W}^{1,#1}(#2)}}
\newcommandtwoopt{\Wone}[2][p][\Omega]{\ensuremath{W^{1,#1}(#2)}}
\newcommandtwoopt{\Wonez}[2][p][\Omega]{\ensuremath{W_0^{1,#1}(#2)}}
\newcommandtwoopt{\Wbarz}[2][p][\Omega]{\ensuremath{W^{1,#1}_0(\clos{#2})}} 
\newcommand{\WoneRd}[1][p]{\ensuremath{\Wone[#1][\RR^d]}}
\newcommand{\BV}{\ensuremath{\mathrm{BV}}}

\newcommand{\emphdef}[1]{\textbf{\boldmath #1\unboldmath}}

\makeatletter
\@ifpackageloaded{xparse}{
\DeclareDocumentCommand{\relcap}{sd!!O{p}O{\Omega}g}{%
  \operatorname{cap}_{#3,#4}%
  \IfValueTF{#5}{%
    \IfValueTF{#3}{%
      \mathopen{\csname#2\endcsname\lparen}#5\mathclose{\csname#2\endcsname\rparen}%
    }{(#5)}%
  }{}%
}
}{}
\makeatother

\newcommand{\rcap}{\relcap}

\makeatletter
\newcommand*{\math@version@bold}{bold}
\DeclareMathOperator{\trace}{\textrm{\upshape\ifx\math@version\math@version@bold\bfseries\else\fi Tr}}
\makeatother

\newcommand{\restrict}[2]{\ensuremath{#1|_{#2}}}
\newcommand{\meas}[2][Lebesgue]{
\ifthenelse{\equal{#1}{Lebesgue}}{\abs{#2}}{{#1(#2)}}
}
\newcommand{\Hm}[1][d-1]{\mathcal{H}^{#1}}

\DeclarePairedDelimiter\norm{\lVert}{\rVert}
\DeclarePairedDelimiter\abs{\lvert}{\rvert}

\makeatletter
\let\old@norm\norm

\@ifpackageloaded{xparse}{
\DeclareDocumentCommand{\smartnorm}{soG{\cdot}}{%
  \IfBooleanTF{#1}{\old@norm{#3}}%
  {%
    \IfValueTF{#2}{\old@norm[#2]{#3}}{\old@norm*{#3}}%
  }%
}
\def\norm{\smartnorm}
}{}
\makeatother

\newcommand{\set}[3][auto]{{
\ifthenelse{\equal{#1}{auto}}{\left\{#2\setmid #3\right\}}{}
\ifthenelse{\equal{#1}{b}}{\bigl\{#2\setmid #3\bigr\}}{}
\ifthenelse{\equal{#1}{B}}{\Bigl\{#2\setmid #3\Bigr\}}{}
}}

\newcommand{\n}[2][n]{$#1$\nobreakdash-\hspace{0pt}#2}

\makeatletter
\newtoks\protect@toks
\def\pdef#1{\protect@toks=\expandafter{\the\protect@toks
 \pdef@#1}%
 \def#1}
\def\pdef@#1{\def#1{\protect#1}}
\newdimen\mex
\pdef\niv{\mathrel{\hbox{\hglue .2\mex
  \vrule \@height 1.33\mex \@width .06\mex
  \vrule \@height .06\mex \@width 1\mex
  \hglue .5\mex}}}

\makeatother

\newlist{normenum}{enumerate}{1}
\setlist[normenum]{label=\textup{\arabic*.},ref=\textup{(\arabic*)},itemsep=0em,topsep=1ex} 

\newlist{arenum}{enumerate}{1}
\setlist[arenum]{label=\textup{(\arabic*)},ref=\textup{(\arabic*)},itemsep=0em,topsep=1ex}

\graphicspath{{./figures/}}

\renewcommand{\phi}{\varphi}

\title{Uniqueness of the approximative trace}

\author[M. Sauter]{\sc Manfred Sauter}
\address{Manfred Sauter\\Institute of Applied Analysis\\Ulm University\\89069 Ulm\\Germany}
\email{manfred.sauter@uni-ulm.de}


\begin{document}

\begin{abstract}
We study the \emph{approximative trace} for individual elements in the Sobolev space
$\Wone$ for $1\le p\le\infty$. This notion of a trace was introduced for $p=2$ in~\cite{AtE2011:DtN} in the setting of general open sets $\Omega\subset\RR^d$. The approximative trace exhibits a curious nonuniqueness phenomenon.
We provide a detailed analysis of this phenomenon based on methods of geometric measure theory
and are able to give very weak geometric conditions that are sufficient for the uniqueness of the approximative trace.
In particular, we prove that the approximative trace is unique on open sets with continuous boundary
and on arbitrary connected domains in $\RR^2$. Furthermore, we provide an example which shows that the uniqueness of the approximative trace depends on $p$.
These results answer several open questions.
\end{abstract}

\begingroup
\renewcommand{\MakeUppercase}[1]{#1}
\maketitle
\endgroup

\section{Introduction}

The purpose of this article is to study the approximative trace of individual elements in the Sobolev space $\Wone$ for $1\le p\le\infty$ on a general open set $\Omega\subset\RR^d$. On the boundary $\partial\Omega$ we consider the $(d-1)$-dimensional Hausdorff measure $\Hm$. If $\Omega$ is a bounded Lipschitz domain, there exists a unique bounded linear operator $\trace\colon\Wone\to L^p(\partial\Omega)$ such that $\trace u = \restrict{u}{\partial\Omega}$ for all $u\in\Wone\cap\Cbar$; the operator $\trace$ is called the \emph{trace operator} on $\Wone$.
This classical result plays a decisive role in analysis, allowing the usual calculus in the form of the divergence theorem and
Green's formulas in the setting of Sobolev functions.

In this paper, however, we consider $\Wone$ on a completely general open set $\Omega$.
It is well-known that the above trace operator does not exist for a domain with sufficiently irregular boundary.
For example, one cannot even expect integrable traces if $\Omega$ has a suitably sharp outward pointing cusp.

Still, it is natural to define approximative traces for individual elements as follows.
For simplicity, let us suppose for a moment that $\Omega$ is bounded and $\Hm(\partial\Omega)<\infty$.
Let $u\in\Wone$ and $\phi\in L^p(\partial\Omega)$. We call $\phi$ an \emph{approximative trace} of $u$ if there
exists a sequence $(u_n)$ in $\Wone\cap\Cbar$ such that $u_n\to u$ in $\Wone$ and $\restrict{u_n}{\partial\Omega}\to\phi$ in $L^p(\partial\Omega)$.
This notion turned out to be useful to treat boundary value problems, see the recent paper~\cite[Section~4]{CHK2015} for several striking applications, or the Dirichlet-to-Neumann operator on rough domains, see~\cite{AtE2011:DtN}.
These are good reasons to study the approximative trace systematically, which is the purpose of this article.

Clearly not every $u\in\Wone$ needs to have an approximative trace, 
but even the uniqueness can fail:
There exist bounded, connected open sets $\Omega\subset\RR^3$ such that the zero function in $\Wone$ has a multitude of nontrivial approximative traces; see~\cite[Example~4.3]{AW03} and~\cite[Section 3, last paragraph on p.\,941]{BG2010} for two examples. This lack of uniqueness is a most curious phenomenon and difficult to understand geometrically.

The \emph{uniqueness of the approximative trace}, i.e.\ that the zero function in $\Wone$ has only the trivial approximative trace, can be reformulated in terms of the closability of the Robin Dirichlet form. In
fact, it was in this context, inspired by an inequality due to Maz$'$ya~\cite[Corollary~2 in Section~6.11.1]{Maz2011},
 that the question of uniqueness occurred in~\cite{Dan2000:robin-bvp} and~\cite{AW03}.
While the lack of uniqueness of the approximative trace was always considered to be rare and pathological, 
up to now there has been no geometric criterion asserting uniqueness that really goes beyond Lipschitz boundary.

We list the four main contributions of this paper.
\begin{normenum}
\item If $\Omega$ has strictly positive Lebesgue density at $\Hm$-a.e.\ $z\in\partial\Omega$, then the approximative trace is unique; see Theorem~\ref{thm:sing-dens-0}.
\item If $\Omega$ has continuous boundary, then the approximative trace is unique; see Theorem~\ref{thm:cont-unique}.
\item If $\Omega\subset\RR^2$ is connected, then the approximative trace is unique; see Corollary~\ref{cor:uniq-two-dim}.
\item We present an example of a connected domain $\Omega\subset\RR^3$ where the approximative trace is not unique for small $p$, but unique for $p$ large; see~Example~\ref{ex:p-dep}. We additionally arrange that for certain values of $p$ the approximative trace is not unique even though every element of $\Wone$ has an approximative trace; see Example~\ref{ex:p-conj-ate}.
\end{normenum}
The first item constitutes a weak measure geometric criterion for the uniqueness of the approximative trace. The results in the second and third item are based on this criterion, but both of them additionally require sophisticated tools from geometric measure theory. In particular, for the third item we need a description of indecomposable sets of finite perimeter in two dimensions from~\cite{ACMM2001}.
Both the second and third items were conjectures in circulation for several years which are now confirmed.
The construction of the example is based on a uniform continuity property of functions in $\Wone$ for $p>d$ established in~\cite{BS2001}.

Despite being a recent notion, the approximative trace is already embedded in a rich theory. 
As we pointed out above, it is intimately connected to the Robin boundary value problem and the Dirichlet-to-Neumann operator. Moreover, there exists an associated notion of capacity, the \emph{relative capacity} as introduced and studied in~\cite{AW03} and~\cite{Bie09:relcap-usem,Bie09:cap-mod}.
The space of elements with approximative trace zero always lies between the spaces $\Wonez$ and $\Wbarz = \{ u\in \Wone : \text{the extension of $u$ by $0$ is in $\WoneRd$}\}$, see~\cite[Chapter~7]{MS2013:thesis}, 
and the stability of the Dirichlet problem is characterised by the coincidence of the three spaces~\cite[Theorem~9]{Hed2000:stab}.

We close this introductory section with a brief outline of the paper. 
In Section~\ref{sec:not-pre} we introduce notation and preliminary results.
We then characterise the approximative traces of the zero function in $\Wone$ in Section~\ref{sec:app-trace-zero} using lattice theory and the notion of the relative capacity.
In Section~\ref{sec:geo-cond} we prove our measure geometric criterion for the uniqueness of the approximative trace and establish uniqueness if $\Omega$ has continuous boundary.
The following Section~\ref{sec:uniq-2d} features the uniqueness of the approximative trace if $\Omega\subset\RR^2$ is connected.
Finally in Section~\ref{sec:ex-app} we present an example where the uniqueness of the approximative trace depends on $p$, point out connections to Maz$'$ya and Burago's \emph{rough trace} and hint on some applications for our results.

\section{Notation and preliminaries}\label{sec:not-pre}

If not explicitly stated otherwise, we consider a general nonempty open set $\Omega\subset\RR^d$. 
In particular, we do not assume that $\Omega$ is bounded, connected, has finite Lebesgue measure or
has a boundary with finite Lebesgue or Hausdorff measure.
We denote the topological boundary of $\Omega$ by $\Gamma$.
For convenience, the function spaces considered in the following are supposed to be real 
since lattice theoretic arguments are used later on.

We define the \emphdef{locally finite part} of $\Gamma$ by
\[
    \Gamma_\loc := \set{z\in\Gamma}{\text{there exists an $r>0$ such that $\Hm(\Gamma\cap B(z,r))<\infty$}}
\]
and set $\Gamma_\nloc := \Gamma\setminus\Gamma_\loc$.
Then $\Gamma_\loc$ is $\sigma$-compact and relatively open in $\Gamma$, and $(\Gamma_\loc,\Borel(\Gamma_\loc),\Hm)$ is a locally finite, $\sigma$-finite, Borel regular measure space. Moreover, if $d>1$ then this space is atomless by~\cite[Exercise~264\,Yg]{Frem2003:vol2}.

We denote by $\Loneloc=\Loneloc[\Gamma,\Borel(\Gamma),\Hm]$ the vector space of \emphdef{locally integrable functions} on $\Gamma$, 
where we identify functions that agree $\Hm$-a.e.\ on $\Gamma$. 
Note that if $\varphi\in\Loneloc$, then $\varphi=0$ $\Hm$-a.e.\ on $\Gamma_\nloc$.
Therefore one can naturally identify $\Loneloc$ 
with $\Loneloc[\Gamma_\loc,\Borel(\Gamma_\loc),\Hm]$. 
Moreover, if $A\subset\Gamma$ is a Borel set, then we consider $\Loneloc[A]$ as a subspace of $\Loneloc$ in the obvious way after extending functions by zero.
We equip $\Loneloc$ with the locally convex topology induced by the family of seminorms 
$\norm{\cdot}_K\colon\varphi\mapsto\norm{\varphi\one_K}_1$ for all compact $K\subset\Gamma_\loc$.
It is easily observed that $\Loneloc$ is completely metrizable, and hence a Fr\'echet space.
In the following we let $Y=\Loneloc$ if not explicitly specified otherwise.

\begin{definition}
For $1\le p\le\infty$, let $u\in\Wone$ and $\varphi\in Y$. Then $\varphi$ is called an \emphdef{approximative trace} of $u$ in $\Wone$ if
there exists a sequence $(u_n)$ in $\Wone\cap\Cbar$ such that $\restrict{u_n}{\Gamma}\in Y$ for all $n\in\NN$ and
\[
    \text{$u_n\to u\in\Wone$}\quad\text{and}\quad\text{$\restrict{u_n}{\Gamma}\to\varphi$ in $Y$.}
\]
Moreover, we define the 
set 
\[
    Z^p := \set{\varphi\in Y}{\text{$\varphi$ is an approximative trace of the zero function in $\Wone$}}.
\]
\end{definition}
In~\cite{AtE2011:DtN,Dan2000:robin-bvp,AW03} only the Hilbert space case $p=2$ 
was considered with approximative traces in the space $Y=\Ltwo[\Gamma]$. 
In~\cite{AW03}, however, also measures different from $\Hm$ were admitted on $\Gamma$.

It is easily observed that in general not every element of $\Wone$ has an approximative trace.
To this end, let us introduce the space
\[
    \Wt = \clos{\Wone\cap\Cbar},
\]
where the closure is taken in $\Wone$.
In general, $\Wt$ is a proper closed subspace of $\Wone$; see~\cite{Kolsrud1981:example} for such an example where $\Omega$ is topologically regular and~\cite[Section~1]{OFar97:sob-approx} for a discussion of related results.
Clearly only elements of $\Wt$ can possibly have an approximative trace.
Yet, if $\Omega$ has a continuous boundary, one has $\Wone=\Wt$ for all $1\le p<\infty$ due to the following result,
where the second statement follows from an inspection of the proof.
\begin{proposition}[see~{\cite[Theorem~V.4.7]{EE87}}]\label{prop:cont-smooth-dense}
Let $\Omega\subset\RR^d$ be open with continuous boundary and $1\le p<\infty$. Then $\{\restrict{u}{\Omega} : u\in\Cinfc[\RR^d]\}$ is dense in $\Wone$.
Moreover, if $u\in\Wone$ with $u\ge 0$ a.e.\ on $\Omega$, then there exists a sequence $(u_n)$ in $\Cinfc[\RR^d]$ such
that $\restrict{u_n}{\Omega}\to u$ in $\Wone$ and $u_n\ge 0$ for all $n\in\NN$.
\end{proposition}
Moreover, if $\Omega$ has a sufficiently sharp outward pointing cusp with tip at $0$, then
the restriction of an element of $\Wone\cap C(\clos{\Omega}\setminus\{0\})$ to the boundary does not need to be 
locally integrable at the tip of the cusp. Such an example is given in~\cite[Example~9.1]{AtE2011:DtN}.
This suggests that in such a case it would be more natural to consider weighted local integrability of the trace, where
the weight depends on the local geometry of $\Omega$ at the boundary. We will not pursue this here and
always consider the measure $\Hm$ without an additional density on $\Gamma$.

The phenomenon that is of foremost interest here is that in general the zero function in $\Wone$ may have nontrivial approximative traces; 
in other words, $Z^p$ does not need to be trivial.
If $Z^p$ is trivial, we say that the \emphdef{approximative trace is unique} in $\Wone$,
and if $Z^p$ is not trivial, then we say that the \emphdef{approximative trace is not unique} in $\Wone$.
In Figure~\ref{fig:forest} a particularly simple example of a domain in $\RR^2$ is depicted, where the approximative trace is not unique provided the size of the balls decreases sufficiently quickly towards the line segment on the left.

We describe the uniqueness of the approximative trace in $\Wone$ in another way.
Consider the operator 
\[
    T_0\colon \{ u\in\Wone\cap \Cbar : \restrict{u}{\Gamma}\in\Loneloc\}\to\Loneloc
\]
given by $T_0 u=\restrict{u}{\Gamma}$. Then the approximative trace is unique in $\Wone$ if and only if $T_0$ is a closable operator in $\Wone\times\Loneloc$.
\begin{figure}
\centering
\includegraphics{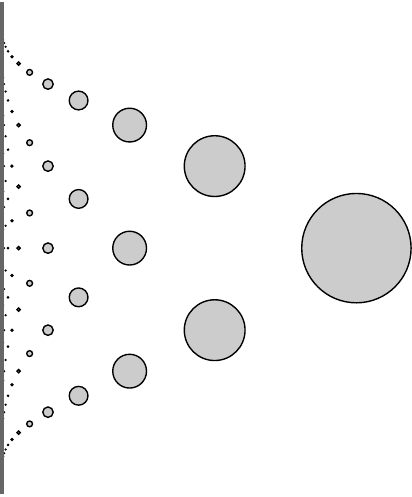}
\caption{An example of an open set $\Omega\subset\RR^2$ where the approximative trace is not unique.}
\label{fig:forest}
\end{figure}

\begin{remark}\label{rem:unique-pinf}
It is easily seen that the approximative trace is always unique in $\Wone[\infty]$. In fact, if $(u_n)$ is a sequence in $\Wone[\infty]\cap\Cbar$ that converges to $0$ in $\Wone[\infty]$, then $(u_n)$ converges uniformly to $0$ on $\clos{\Omega}$. So $Z^\infty = \{0\}$.
\end{remark}

We begin with a basic lattice theoretic property.
\begin{lemma}\label{lem:lat-approx-tr}
The set $Z^p$ is a closed vector sublattice of $\Loneloc$.
\end{lemma}
\begin{proof}
It is readily seen that $Z^p$ is a vector subspace of $\Loneloc$.

Suppose the Fr\'echet space $\Loneloc$ is equipped with the metric $d$.
Let $\phi\in\Loneloc$ and $\phi_n\in Z^p$ for all $n\in \NN$ be such that $\phi_n\to\phi$ in $\Loneloc$.
Let $u_n\in\Wone\cap\Cbar$ be such that $\restrict{u_n}{\Gamma}\in\Loneloc$, $\norm{u_n}_{\Wone}\le\frac{1}{n}$, $d(\phi_n,\restrict{u_n}{\Gamma})\le\frac{1}{n}$ for all $n\in\NN$.
Then $u_n\to 0$ in $\Wone$ and $d(\phi,\restrict{u_n}{\Gamma})\le d(\phi,\phi_n)+d(\phi_n,\restrict{u_n}{\Gamma})\to 0$. Hence $\phi\in Z^p$ and $Z^p$ is closed in $\Loneloc$.

It remains to show that $Z^p$ is a sublattice.
As $\Loneloc$ is a vector lattice, is suffices to show that $\phi\vee\psi\in Z^p$ for all $\phi,\psi\in Z^p$.
Let $\phi,\psi\in Z^p$ and $(u_n)$, $(v_n)$ be sequences in $\Wone\cap\Cbar$ such that $u_n\to 0$ and $v_n\to 0$ in $\Wone$ and
$\restrict{u_n}{\Gamma}\to\phi$ and $\restrict{v_n}{\Gamma}\to\psi$ in $\Loneloc$.
After passing to a subsequence we may in addition suppose that
$\restrict{u_n}{\Gamma}\to\phi$ and $\restrict{v_n}{\Gamma}\to\psi$ $\Hm$-a.e.\ on $\Gamma$.
Then $u_n\vee v_n\to 0$ in $\Wone$ and $u_n\vee v_n\to \phi\vee\psi$ $\Hm$-a.e.\ on $\Gamma$.
Let $K\subset\Gamma_\loc$ be compact. Then, again passing to a subsequence, we may suppose that there exists a $g\in\Lone[K]$ such that $\max\{\abs{u_n},\abs{v_n}\}\le g$ on $K$ for all $n\in\NN$.
By Lebesgue dominated convergence we obtain $u_n\vee v_n\to\phi\vee\psi$ in $\Lone[K]$.
As $\Gamma_\loc$ is $\sigma$-compact, a diagonal sequence argument yields $\phi\vee\psi\in Z^p$.
\end{proof}

\begin{remark}
In~\cite[Section~4]{CHK2015} the approximative trace is used in a slightly more general setting.
The functions there are elements of $L^q(\Omega)$ with weak partial derivatives in $L^p(\Omega)$, and the approximative trace is defined in $Y=L^r(\partial\Omega)$. While in~\cite[Definition~4.2]{CHK2015} all values $1\le p,q,r\le\infty$ are admitted, in the applications given there only a certain range of parameters governed by Maz$'$ya's inequality~\cite[Corollary~2 in Section~6.11.1]{Maz2011} plays a role.

In~\cite[Remark~4.3.(b)]{CHK2015} the authors express that they expect the uniqueness of the approximative trace to depend nontrivially on $r$ in general. This expectation is unfounded, which we can see as follows. We use the notation from~\cite{CHK2015}.
Firstly, we exclude the (mostly trivial) cases when $p=\infty$, $q=\infty$ or $r=\infty$.
Of course, individual elements of $W^1_{p,q}(\Omega)$ can have a unique approximative trace in $Y=L^r(\partial\Omega)$, without admitting an approximative trace in $Y=L^s(\partial\Omega)$ for an $s\ne r$. We show that whether an element has more than one approximative trace in $L^r(\partial\Omega)$ is independent of $r$.
Because of linearity, it suffices to consider approximative traces of the zero function in $W^1_{p,q}(\Omega)$.
In the following let $(u_n)$ be a sequence in $W^1_{p,q}(\Omega)$ such that $u_n\to 0$ in $W^1_{p,q}(\Omega)$.
Suppose that $\restrict{u_n}{\partial\Omega} \to \phi\ne 0$ in $L^r(\partial\Omega)$.
Then obviously the approximative trace in $W^1_{p,q}(\Omega)$ for $Y=\Loneloc[\partial\Omega]$ is not unique.
Conversely, suppose that $\restrict{u_n}{\partial\Omega} \to \phi\ne 0$ in $\Loneloc[\partial\Omega]$.
By a straightforward truncation and cut-off argument we may assume in addition that $\Omega$ is bounded, $\Hm(\partial\Omega)<\infty$, $0\le u_n\le 1$, $0\le\phi\le 1$ and $\restrict{u_n}{\partial\Omega} \to \phi$ in $L^1(\partial\Omega)$.
Moreover, after passing to a subsequence we may assume that we have pointwise convergence $\Hm$-a.e.\ on the boundary. By Lebesgue dominated convergence, we obtain
$\restrict{u_n}{\partial\Omega} \to \phi$ in $L^r(\partial\Omega)$. As $\phi\ne0$, the approximative trace in $W^1_{p,q}(\Omega)$ for $Y=L^r(\partial\Omega)$ is not unique.
In other words, the approximative trace for $Y=L^r(\partial\Omega)$ is unique if and only if this is the case for $Y=\Loneloc[\partial\Omega]$.
\end{remark}

\section{Approximative traces of the zero function}\label{sec:app-trace-zero}

In this section we describe the approximative traces of the zero function in $\Wone$.
We extend the presentation from~\cite{AW03} to the $p$-dependent setting, where we always choose the Hausdorff measure $\Hm$ on $\Gamma$.
Since some modifications are required in our setting, we give the details.
The case $p=\infty$ can be excluded by Remark~\ref{rem:unique-pinf}.
We fix $1\le p<\infty$ and an open set $\Omega\subset\RR^d$.

For the following lemma we use the argument from~\cite[Lemma~4.14]{AtE12:sect-form}.
In~\cite[Lemma~3.4]{Dan2000:robin-bvp} a different proof is given. 
In both of these references the treatment is focused on the case $p=2$.
\begin{lemma}\label{lem:Xs-ind-mult}
Suppose that $\varphi\in Z^p$ and $K\subset\Gamma_\loc$ is compact.
Then $\one_K\varphi\in Z^p$.
\end{lemma}
\begin{proof}
Let $\varphi\in Z^p$ and $\psi\in\Cinfc[\RR^d]$. 
Suppose $(u_n)$ is a sequence in $\Wone\cap\Cbar$ such that $u_n\to 0$ in $\Wone$ and $\restrict{u_n}{\Gamma}\to\phi$ in $\Loneloc$.
Then $\psi u_n\to 0$ in $\Wone$ and $\restrict{(\psi u_n)}{\Gamma}\to\psi\phi$ in $\Loneloc$.
Hence $\psi\phi\in Z^p$ for all $\psi\in\Cinfc[\RR^d]$.

As $\dist(K,\Gamma_\nloc)>0$, there exists a bounded open set $V\subset\RR^d$ such that $K\subset V$ and
$\dist(V,\Gamma_\nloc)>0$.
Moreover, there exists an $\eta\in\Cinfc[V]$ such that $0\le\eta\le 1$ and $\eta=1$ on $K$.
Let $M := \clos{V}\cap\Gamma$. Then $M$ is compact and $M\subset\Gamma_\loc$.
By the Stone--Weierstrass theorem $\{\restrict{\psi}{M} : \psi\in\Cinfc[\RR^d]\}$ is dense in $C(M)$.
As $C(M)$ is dense in $\Lone[M]$, there exists a sequence $(\psi_n)$ in $\Cinfc[\RR^d]$ such that $\restrict{\psi_n}{M}\to\one_K$ in
$\Lone[M]$. Then $\restrict{(\eta\psi_n)}{\Gamma}\to\one_K$ in $\Lone[\Gamma]$.
It follows from the first part of the proof that $\eta\psi_n\phi\in Z^p$ for all $n\in\NN$.
As $Z^p$ is closed in $\Loneloc$ and $\eta\psi_n\phi\to\one_K\phi$ in $\Lone[\Gamma]$, one obtains $\one_K\phi\in
Z^p$.
\end{proof}

The next proposition and its corollary extend~\cite[Lemma~3.4 and Proposition~3.3]{Dan2000:robin-bvp}
to our setting.
\begin{proposition}\label{prop:Xs-ideal}
One has $\Linf[\Gamma_\loc] Z^p \subset Z^p$.
Moreover, $Z^p$ is a closed lattice ideal of $\Loneloc$.
\end{proposition}
\begin{proof}
The first statement follows from Lemma~\ref{lem:Xs-ind-mult} by a 
monotone class argument.
In fact, consider the set
\[
    \cM := \{ h\in\Linf[\Gamma_\loc]:h\phi\in Z^p\text{ for all }\phi\in Z^p \}.
\]
Clearly $\cM$ is a vector subspace of $\Linf[\Gamma_\loc]$.
Moreover, $\one_K\in\cM$ for all compact sets $K\subset\Gamma_\loc$ by Lemma~\ref{lem:Xs-ind-mult}.
Let $h\in\Linf[\Gamma_\loc]$ and $(h_n)$ be a sequence in $\cM$ such that $h_n\ge 0$ and $h_n\uparrow h$ as
$n\to\infty$. Suppose $\phi\in Z^p$. Then it follows from Lebesgue dominated convergence that
$h_n\phi\to h\phi$ in $\Loneloc$. As $Z^p$ is closed, one deduces $h\phi\in Z^p$ and therefore $h\in\cM$.
By the monotone class theorem, it follows that $\cM=\Linf[\Gamma_\loc]$.
Therefore $\Linf[\Gamma_\loc] Z^p\subset Z^p$.

We already know that $Z^p$ is a closed sublattice of $\Loneloc$ by Lemma~\ref{lem:lat-approx-tr}.
It remains to show that $Z^p$ is an ideal.
Let $\phi\in Z^p$ and $g\in\Loneloc$ be such that $0\le g\le\phi$.
Define the function $h\in\Linf[\Gamma_\loc]$ by setting 
\[
    h(x):= \begin{cases}
   \frac{g(x)}{\phi(x)} & \text{if $\phi(x)>0$,} \\
   0 & \text{if $\phi(x)=0$.}
   \end{cases}
\]   
Then by the first part of the proof one has $g=h\phi\in Z^p$.
This shows that $Z^p$ is an ideal of $\Loneloc$.
\end{proof}

\begin{corollary}\label{cor:Gamma_s}
There exists a Borel set $\Gamma_\sing^p\subset\Gamma_\loc$ such that $Z^p=\Loneloc[\Gamma_\sing^p]$.
\end{corollary}
The subscript `$\sing$' in the notation stands for `singular'.
\begin{proof}
The claim follows from Proposition~\ref{prop:Xs-ideal} and Schaefer's description of closed ideals in $L^q$
spaces~\cite[Example~2 on p.\,157]{Sch74}. Since Schaefer's description is not given for local $L^q$ spaces,
we shall address this detail.

As $(\Gamma_\loc,\Borel(\Gamma_\loc),\Hm)$ is $\sigma$-finite, it is localisable and does not have infinite
atoms. Clearly $\one_K Z^p$ is a closed ideal in $\Lone[K]$ for every compact $K\subset\Gamma_\loc$.
Let $(K_n)$ be an increasing sequence of compact subsets of $\Gamma_\loc$ with $\bigcup_{n\in\NN} K_n=\Gamma_\loc$.
By~\cite[Example~2 on p.\,157]{Sch74} there exists a Borel set $A_n\subset K_n$ such that $\Lone[A_n]=\one_{K_n}Z^p$
for all $n\in\NN$. Fix $n,k\in\NN$. 
As $\one_{A_{k}\cap K_n\setminus A_n}\in\Lone[A_{k}]\subset Z^p$, it follows that
$\one_{A_{k}\cap K_n\setminus A_n}\in \one_{K_n}Z^p=\Lone[A_n]$. Hence $\Hm(A_{k}\cap K_n\setminus A_n)=0$.
Set $\Gamma_\sing^p=\bigcup_{n\in\NN} A_n$.
It follows that
\[
    \Hm((\Gamma_\sing^p\cap K_n)\symdiff A_n) = \Hm(\Gamma_\sing^p\cap K_n\setminus A_n) \le\sum_{k=1}^\infty\Hm(A_k\cap
K_n\setminus A_n)=0
\]
for all $n\in\NN$. Hence we may replace $A_n$ by $\Gamma_\sing^p\cap K_n$ to obtain $L^1(\Gamma_\sing^p\cap K_n)=\one_{K_n}Z^p$ for all $n\in\NN$.

If $\phi\in Z^p$, then $\one_{K_n}\phi=0$ $\Hm$-a.e.\ on $K_n\setminus A_n$ for all $n\in\NN$.
So $\phi=0$ $\Hm$-a.e.\ on $\Gamma_\loc\setminus\Gamma_\sing^p$. Thus $\phi\in\Loneloc[\Gamma_\sing^p]$.
Conversely, suppose $\phi\in\Loneloc[\Gamma]$ such that $\phi=0$ $\Hm$-a.e.\ on $\Gamma_\loc\setminus\Gamma_\sing^p$.
Then $\one_{K_n}\phi\in\Lone[A_n]=\one_{K_n}Z^p\subset Z^p$ for all $n\in\NN$.
As $Z^p$ is closed and $\one_{K_n}\phi\to\phi$ in $\Loneloc$ as $n\to\infty$ by Lebesgue dominated convergence, one deduces $\phi\in Z^p$.
\end{proof}

Obviously the approximative trace is not unique in $\Wone$ if and only if $\Hm(\Gamma^p_\sing)>0$.
We introduce a localised notion of uniqueness and nonuniqueness of the approximative trace on a part of the boundary.
\begin{definition}
Let $A$ be a Borel subset of $\Gamma_\loc$.
The approximative trace is said to be \emphdef{unique on $A$} if $\Hm(A\cap\Gamma^p_\sing)=0$.
Conversely, the approximative trace is called \emphdef{nonunique on $A$} if $\Hm(A\cap\Gamma^p_\sing)>0$.
\end{definition}

\begin{lemma}\label{lem:nonuni-witness}
Let $\Omega\subset\RR^d$ be open and $A\subset\Gamma_\loc$ a Borel set. 
Suppose that the approximative trace is nonunique on $A$.
Then there exist a compact set $K\subset A$ with $\Hm(K)>0$, $U\subset\RR^d$ open, bounded and smooth with $K\subset U$ and $\dist(U,\Gamma_\nloc)>0$, $\phi\in L^1(\Gamma)$ and a sequence $(u_n)$ in $\Wone\cap\Cbar$ such that
$u_n\to 0$ in $\Wone$, $0\le u_n\le 1$, $\supp u_n\Subset U$, $\restrict{u_n}{\Gamma}\to\phi$ in $L^1(\Gamma)$ and $u_n=1$ on $K$ for all $n\in\NN$.
\end{lemma}
\begin{proof}
Suppose the approximative trace is nonunique on $A$. By Corollary~\ref{cor:Gamma_s} one has $Z^p =
\Loneloc[\Gamma_\sing^p]$ with $\Hm(\Gamma_\sing^p\cap A)>0$. 
As $\Hm$ is $\sigma$-finite on $\Gamma_\loc$, there exists a Borel set $A'\subset\Gamma_\sing^p\cap A$ such that
$0<\Hm(A')<\infty$. Then the $(d-1)$-dimensional Hausdorff measure restricted to $A'$ is a Radon measure
in $\RR^d$ by~\cite[Theorem~1.1.3]{EG92}.

As $\phi' := \one_{A'}\in Z^p$, there exists a sequence $(w_n)$ in $\Wone\cap\Cbar$ such that $w_n\to 0$ in $\Wone$ and
$\restrict{w_n}{\Gamma}\to\phi$ in $\Loneloc$. After passing to a subsequence, we may assume that $w_n\to 1$
$\Hm$-a.e.\ on $A'$. By Egorov's theorem~\cite[Theorem~1.2.3]{EG92} and as $\Hm$ is a Radon measure on $A'$, there exists
a compact set $K\subset A'$ such that $\Hm(K)>0$ and $w_n\to 1$ uniformly on $K$.
By discarding finitely many elements, we may suppose that $w_n\ge\frac{1}{2}$ on $K$ for all $n\in\NN$.

Let $U\subset\RR^d$ be a smooth, bounded open set such that $K\subset U$ and $\dist(U,\Gamma_\nloc)>0$. Take $\eta\in\Cinfc[U]$ such that $0\le\eta\le 1$ and $\eta=1$ on $K$.
It suffices now to note that $\phi:=\eta\phi'\in L^1(\Gamma)$ and the sequence $(u_n)$ in $\Wone\cap\Cbar$ given by $u_n := \eta(0\vee 2w_n\wedge 1)$ have the desired properties.
\end{proof}

\begin{lemma}\label{lem:Gsing-incl}
Let $1\le p\le q\le\infty$. Then $Z^q$ is a subspace of $Z^p$ and $\Hm(\Gamma_\sing^q\setminus\Gamma_\sing^p)=0$.
\end{lemma}
\begin{proof}
Let $K\subset\Gamma_\sing^q$ be compact. It suffices to show that $\one_K\in Z^p$.
There exists a sequence $(u_n)$ in $\Wone[q]\cap\Cbar$ such that $u_n\to 0$ in $\Wone[q]$ and $\restrict{u_n}{\Gamma}\to\one_K$ in $\Loneloc$.
By multiplying with a suitable cut-off function in $\Cinfc[\RR^d]$, we can assume that there exists a compact set $M\subset\RR^d$ such that $\supp u_n\subset M$ for all $n\in\NN$. It follows that $(u_n)$ is contained in $\Wone\cap\Cbar$ and $u_n\to 0$ in $\Wone$. Hence $\one_K\in Z^p$.
\end{proof}

For the remainder of this section, we suppose that $1<p<\infty$
and use the relative capacity and its properties as established in~\cite{Bie09:relcap-usem,Bie09:cap-mod}.
The \emphdef{relative $p$-capacity} of a subset $A\subset\clos{\Omega}$ is defined as
\begin{multline*}
 \rcap{A} = \inf\bigl\{\norm{u}^p_{1,p}\setmid\text{$u\in\Wt$ such that there exists an open $V\subset\RR^d$}\\
     \text{with $A\subset V$ and $u\ge 1$ a.e.\ on $V\cap\Omega$}\bigr\}.
\end{multline*}
If $\rcap(A)=0$, then $A$ is called \emphdef{relatively $p$-polar}, and if a property holds outside of a relatively $p$-polar set, it is said to hold \emphdef{relatively $p$-quasi everywhere}. 
We note that the relative $p$-capacity for $\Omega=\RR^d$ is the usual $p$-capacity.

We exclude the nonreflexive case $p=1$ since it is unclear whether the relative capacity is a Choquet capacity in this case.
We define admissible subsets of $\Gamma_\loc$ as in~\cite[(13)]{AW03}.
\begin{definition}
A Borel set $M\subset\Gamma_\loc$ is called \emphdef{$p$-admissible} if for every 
Borel set $A\subset\Gamma_\loc$ one has that $\rcap(A)=0$ implies $\Hm(M\cap A)=0$.
\end{definition}

By utilising the relative capacity, the following theorem allows to select a finer version of the set $\Gamma^p_\sing$ that enjoys additional properties.
The set $\Gamma^p_\sing$ as chosen above does not need to be relatively $p$-polar in general.
\begin{theorem}\label{thm:char-Xs-cap}
Let $\Omega\subset\RR^d$ be open.
There exists a Borel set $S\subset\Gamma_\loc$ such that $Z^p = \Loneloc[S]$,
$\rcap(S)=0$ and $\Gamma_\loc\setminus S$ is $p$-admissible.
Moreover, if $S'\subset\Gamma_\loc$ is a relatively $p$-polar Borel set, then $\Gamma_\loc\setminus S'$ is $p$-admissible
if and only if $\Hm(S'\symdiff S)=0$.
\end{theorem}

We split the proof of Theorem~\ref{thm:char-Xs-cap} into several parts.
\begin{proposition}\label{prop:S-polar}
There exists a relatively $p$-polar Borel set $S\subset\Gamma_\loc$ 
such that $\Gamma_\loc\setminus S$ is $p$-admissible.
\end{proposition}
\begin{proof}
The proof of~\cite[Proposition~3.6]{AW03} works without change.
\end{proof}

In the remainder of this section let $S$ denote a set as in Proposition~\ref{prop:S-polar}.
\begin{lemma}\label{lem:S-subset-Gs}
One has $\Hm(S\setminus\Gamma^p_\sing)=0$.
\end{lemma}
\begin{proof}
This uses part of the argument in the proof of~\cite[Theorem~3.7]{AW03}.

Let $(K_n)$ be an increasing sequence of compact subsets of $S$ such that $\Hm(S\setminus\bigcup_{n\in\NN} K_n)=0$.
As $\rcap{K_n}=0$, by~\cite[Proposition~3.5]{Bie09:relcap-usem} we can find a sequence $(u_n)$ in $\Wone\cap\Cbar$ 
such that $0\le u_n\le 1$, $u_n=1$ on $K_n$, $u_n=0$ on $\Gamma_\nloc$ and $\norm{u_n}_{\Wone}\le\frac{1}{n}$.
After going to a subsequence, we may assume that $u_n\to 0$ relatively $p$-quasi everywhere.
As $\Gamma_\loc\setminus S$ is $p$-admissible, it follows that $u_n\to 0$ pointwise $\Hm$-a.e.\ on $\Gamma_\loc\setminus S$. Hence $u_n\to\one_S$ pointwise $\Hm$-a.e.\ on $\Gamma$.
By Lebesgue dominated convergence, one deduces that $\restrict{u_n}{\Gamma}\to\one_S$ in $\Loneloc$.
This shows that $\one_S\in Z^p=\Loneloc[\Gamma_\sing^p]$; consequently, $\Hm(S\setminus\Gamma_\sing^p)=0$.
\end{proof}

\begin{lemma}\label{lem:Gs-subset-S}
One has $\Hm(\Gamma^p_\sing\setminus S)=0$.
\end{lemma}
\begin{proof}
The function $\one_{\Gamma_\sing^p}$ is an approximative trace of the zero function in $\Wone$.
So there exists a sequence $(u_n)$ in $\Wone\cap\Cbar$ such that $\restrict{u_n}{\Gamma}\to\one_{\Gamma_\sing^p}$ in $\Loneloc$ and $u_n\to0$ in $\Wone$.
After passing to a subsequence, we may suppose that $\restrict{u_n}{\Gamma}\to\one_{\Gamma_\sing^p}$ pointwise $\Hm$-a.e.\ on $\Gamma$ and $u_n\to 0$ relatively $p$-quasi everywhere on $\clos{\Omega}$.
As $\Gamma_\loc\setminus S$ is $p$-admissible, we obtain that $u_n\to 0$ pointwise $\Hm$-a.e.\ on $\Gamma_\loc\setminus S$.
Hence $\Hm(\Gamma_\sing^p\setminus S)=\Hm(\Gamma_\sing^p\cap\Gamma_\loc\setminus S)=0$.
\end{proof}

\begin{proof}[Proof of Theorem~\ref{thm:char-Xs-cap}]
It follows from Lemmas~\ref{lem:S-subset-Gs} and~\ref{lem:Gs-subset-S} that $\Hm(\Gamma_\sing^p\symdiff S)=0$.
The same holds for $S'\subset\Gamma_\loc$ if $S'$ is relatively $p$-polar such that $\Gamma_\loc\setminus S'$ is $p$-admissible.
Conversely, if $S'\subset\Gamma_\loc$ is relatively $p$-polar with $\Hm(S\symdiff S')=0$,
it is obvious that $S'$ is $p$-admissible.
\end{proof}

The relative $p$-capacity is equivalent to the $p$-capacity as long as one only considers subsets that stay a fixed positive distance away from the boundary. Moreover, if $\Omega$ is Lipschitz and thus has the extension property, then the relative $p$-capacity is equivalent to the $p$-capacity on the whole of $\clos{\Omega}$.
So it is clear that the relative $p$-capacity is $p$-dependent and that there are compact sets that are relatively polar for some small $p$, but not for some larger $p$, since this is well-known for the $p$-capacity; see for example~\cite[Theorem~5.5.1]{AH96}. Moreover, one can readily give examples where a compact subset $K$ of the boundary is relatively $p$-polar only for certain values of $p$ by considering $\Wone[p]$ for $\Omega=\RR^d\setminus K$.
In the present context it is more interesting, however, to give an example where the corresponding part of the boundary is not an $\Hm$-nullset. This will be done in Example~\ref{ex:p-dep}.

\section{A geometric criterion for the uniqueness of the approximative trace}\label{sec:geo-cond}

The results of this section imply that the approximative trace is unique if $\Omega$ has continuous boundary,
which was stated as an open problem in~\cite[Section~9]{AtE2011:DtN}. 
In fact, the main result is Theorem~\ref{thm:sing-dens-0}, which gives a criterion for the uniqueness of the approximative trace that applies in the case of continuous boundary.

At the core of the proof one needs to control the behaviour locally on the locally finite part of the boundary through the behaviour of the Sobolev function in the interior.
While the setup of the proof of Theorem~\ref{thm:sing-dens-0} is very natural, it requires some technical results from geometric measure theory.

We first recall a few facts about densities, the measure theoretic boundary and about $(d-1)$-rectifiable sets.
The ambient space will always be $\RR^d$ in the following. So $B(x,r)$ denotes the open Euclidean ball about $x$ with radius $r$ in $\RR^d$ and $\mathbb{S}^{d-1}:=\partial B(0,1)$.
We denote the Lebesgue measure in $\RR^d$ by $\meas{\mathord{\cdot}}$.
For a Borel measurable set $A\subset\RR^d$ and a point $x\in\RR^d$ the \emphdef{upper} and \emphdef{lower
densities} of $A$ at $x$ are defined by
\[
    \overline{D}(A,x) := \limsup_{r\to 0+}\frac{\meas{A\cap B(x,r)}}{\meas{B(x,r)}}
\]
and
\[
    \underline{D}(A,x) := \liminf_{r\to 0+}\frac{\meas{A\cap B(x,r)}}{\meas{B(x,r)}},
\]
respectively. If upper and lower densities agree, their common value is called the \emphdef{density} of $A$ at $x$ and denoted by $D(A,x)$. For all $t\in[0,1]$ we write
\[
    A^t := \{x\in\RR^d: D(A,x)=t\}
\]
for the set of points where $A$ has density $t$. It is easily verified that the sets $A^t$ are Borel, see e.g.~\cite[Lemma~2.1]{Fal86}.
The \emphdef{measure theoretic boundary} of $A$ is defined as $\mbdy A :=\RR^d\setminus(A^0\cup A^1)$.
Clearly $\mbdy A\subset\partial A$.
Similarly we define the \emphdef{$(d-1)$-dimensional density} of $A$ at $x$ by
\[
    \Theta(A,x) := \lim_{r\to 0+}\frac{\Hm(A\cap B(x,r))}{\omega r^{d-1}}
\]
provided the limit exists, and analogously the upper and lower $(d-1)$-dimensional densities $\overline{\Theta}(A,x)$ and $\underline{\Theta}(A,x)$; here
\[
    \omega := \omega(d-1) := \frac{\pi^{(d-1)/2}}{\Gamma\bigl((d+1)/2\bigr)} 
\]
is the volume of the unit ball in $\RR^{d-1}$. 
A Borel measurable subset $A\subset\RR^d$ is called \emphdef{rectifiable} (or, more specifically, countably $(d-1)$-rectifiable) if there exists a sequence $(f_k)$ of Lipschitz functions
$f_k\colon\RR^{d-1}\to\RR^d$ such that 
\[
    \Hm\Bigl(A\setminus\bigcup_{k\in\NN} f_k(\RR^{d-1})\Bigr) = 0.
\]
Conversely, $A$ is called \emphdef{purely unrectifiable} if for any Lipschitz function 
$f\colon\RR^{d-1}\to\RR^d$ one has $\Hm(A\cap f(\RR^{d-1}))=0$.
Rectifiable sets of finite Hausdorff measure are the measure theoretic analogue of smooth manifolds.
They exhibit a plethora of additional density and differentiability properties.
The ones that we will need later on are collected in the following lemma.

\begin{lemma}\label{lem:reg-rect}
Let $K\subset\RR^d$ be a rectifiable Borel set with $\Hm(K)<\infty$. 
Then $\Hm$-a.e.\ $z\in K$ has the following properties:
\begin{arenum}
\item\label{en:rect-dens1}
$\Theta(K,z)=1$.
\item\label{en:rect-tanplane}
There exists a unique affine $(d-1)$-plane $V$ through $z$ with normal $\nu\in\mathbb{S}^{d-1}$ such that for all $\eps>0$ one has
\[
    \lim_{r\to 0+}\frac{\Hm(K\cap B(z,r)\setminus N_{\eps r}(V))}{r^{d-1}}=0,
\]
where $N_{\eps r}(V) := \{x\in\RR^d : \dist(x,V)\le\eps r\}$.
\item\label{en:rect-lipgraph}
For all $\eps>0$ there exist a Borel set $D\subset V$ with $z\in D$ and an $\eps$-Lipschitz function $f\colon D\to\linspan\{\nu\}$
such that $\underline{\Theta}(D,z)\ge1-\eps$, $f(z)=0$, $\graph f \subset K$ and $\Theta(K\setminus\graph f,z)=0$.
\end{arenum}
\end{lemma}
\begin{proof}
The density property~\ref{en:rect-dens1} is established in~\cite[Theorem~16.2]{Mat95},
while the existence of a unique approximate tangent $(d-1)$-plane for $\Hm$-a.e.\ $z\in K$ as in~\ref{en:rect-tanplane} follows from~\cite[Theorem~15.19]{Mat95}.

Let $\eps>0$ and $\eps'\in(0,\frac{\eps}{3})$ be sufficiently small.
Due to~\cite[Proposition~2.76]{AFP00} 
and~\cite[Theorem~6.2\,(2)]{Mat95}, we may suppose without loss of generality that $K$ is the graph of an \n[\eps']Lipschitz function $g\colon\RR^{d-1}\to\RR$ on a Borel set $\tilde{D}\subset\RR^{d-1}\wedgeq\RR^{d-1}\times\{0\}$.
Then for $\Hm$-a.e.\ $z\in K$ with $z=x+g(x)$ for an $x\in\tilde{D}$ the first two properties hold, $\Theta(\tilde{D},x)=1$ and the function $g$ is differentiable at $x$ by Rademacher's theorem. Fix such a $z=x+g(x)\in K$. 
Then by~\cite[paragraph after Definition~2.60]{AFP00}
the set $K$ is contained in the graph of an $(3\eps')$-Lipschitz function $f\colon V\to\linspan\{\nu\}^\perp$, where $V=z+\{(h,g'(x)h) : h\in\RR^{d-1}\}$ is both the unique approximate tangent plane of $K$ at $z$ and the classical tangent plane of $f$ at $x$, cf.~\cite[Remark~2.84]{AFP00}. 
Let $D:=\{x'\in V: x'+f(x')\in K\}$. Then $z\in D$ and $f(z)=0$.
Using that $f$ is $(3\eps')$-Lipschitz, it follows from~(1) that $\underline{\Theta}(D,z)\ge 1-\eps$ provided $\eps'$ is sufficiently small.
\end{proof}
\begin{remark}
In the following we prefer to use cubes instead of balls for the densities and properties in Lemma~\ref{lem:reg-rect}. More specifically, for a $z\in K$ with the three properties in Lemma~\ref{lem:reg-rect}, we consider cubes centred at $z$ in a fixed local orthogonal coordinate system $(b_1,\dots, b_{d-1},\nu)$ at $z$, i.e.\ $b_1,\dots,b_{d-1}$ are orthonormal and span $V-z$.
It is clear that both~\ref{en:rect-tanplane} and~\ref{en:rect-lipgraph} stay valid if balls are replaced by these cubes.
Naturally, one has to adapt the scaling accordingly and use $2^{d-1}$ instead of $\omega$ in the $(d-1)$-dimensional densities.
Since in~\ref{en:rect-lipgraph} one may use any $\eps>0$, it follows by an easy estimate that also~\ref{en:rect-dens1} holds accordingly for the aforementioned cubes.
\end{remark}

\begin{lemma}The set $\Gamma_\loc$ can be written as a disjoint union of two Borel sets $\Gamma_\rect$ and $\Gamma_\urect$ where
$\Gamma_\rect$ is rectifiable and $\Gamma_\urect$ is purely unrectifiable.
Moreover, one has $\Theta(\Gamma_\urect,z)=0$ for $\Hm$-a.e.\ $z\in\RR^d\setminus\Gamma_\urect$.
\end{lemma}
\begin{proof}
As $\Hm$ is locally finite and $\sigma$-finite on $\Gamma_\loc$,
the claims follow immediately from~\cite[Theorem~15.6]{Mat95} or~\cite[Theorem~5.7]{DeLe2008}, and~\cite[Theorem~6.2]{Mat95}.
\end{proof}

Next we briefly introduce sets of locally finite perimeter. By Federer's criterion~\cite[Theorem~4.5.11]{Federer1969} a Borel measurable set $A\subset\RR^d$ has \emphdef{locally finite perimeter} in $\RR^d$ if and only if $\Hm(\mbdy A\cap K)<\infty$ for all compact $K\subset\RR^d$. Moreover, $A$ has \emphdef{finite perimeter} in $\RR^d$ if and only if $\Hm(\mbdy A)<\infty$.
The theory of sets of finite perimeter allows to partition $\Gamma_\loc$ according to the density of $\Omega$ in the following way.
\begin{lemma}\label{lem:dens-Gloc}
Let $\Omega$ be open. Then up to an $\Hm$-nullset one has $\Gamma_\loc=\Gamma_\loc^0\cup\Gamma_\loc^1\cup\Gamma_\loc^{1/2}$, where $\Gamma_\loc^t := \Gamma_\loc\cap\Omega^t$. Moreover, the part $\Gamma_\loc^{1/2}$ is rectifiable.
\end{lemma}
\begin{proof}
Let $K\subset\Gamma_\loc$ be compact.
Let $U\subset\RR^d$ be open, bounded, smooth such that $K\subset U$ and $\dist(\clos{U},\Gamma_\nloc)>0$.
Then $U\cap\Omega$ has finite perimeter as $\partial(U\cap\Omega)\subset\partial U\cup(\clos{U}\cap \Gamma_\loc)$ has finite $\Hm$-measure. Observe that $\overline{D}(U\cap\Omega,z)=\overline{D}(\Omega,z)$ 
and $\underline{D}(U\cap\Omega,z)=\underline{D}(\Omega,z)$ for all $z\in K$.
By~\cite[Theorems~3.59 and~3.61]{AFP00} it follows that $\smash{D(\Omega,z)\in\{0,1,\frac{1}{2}\}}$ for $\Hm$-a.e.\ $z\in K$ and that $K\cap\Omega^{1/2}$ is rectifiable.
The claim follows as $\Gamma_\loc$ is $\sigma$-compact.
\end{proof}
Consequently, up to an $\Hm$-nullset, the purely unrectifiable part $\Gamma_\urect$ is contained in $\Gamma^0_\loc\cup\Gamma^1_\loc$.

The following auxiliary lemma can alternatively be inferred from the results in~\cite[p.\,66--67]{Mor66}.
\begin{lemma}\label{lem:reg-lines}
Let $1\le p<\infty$, $u\in\Wone\cap\Cbar$ and $\nu\in\mathbb{S}^{d-1}$. 
Then for $\Hm$-a.e.\ line $L$ in direction $\nu$
one has the following property: If $a,b\in L$ with $a\ne b$ and such that the open line segment $(a,b)\subset\Omega$, then $u$ is absolutely continuous on the compact line segment $[a,b]$.
Moreover, the classical partial (or directional) derivatives of $u$ agree with the distributional ones a.e.\ in $\Omega$.
\end{lemma}
\begin{proof}
By~\cite[Theorem~2.2.2]{Ziemer1989} one may, after rotation, suppose that $\nu=e_1$.
Let $\tilde{u}\colon\Omega\to\RR$ be a representative of $u$ that is locally absolutely continuous on $\Hm$-a.e.\ 
line parallel to the coordinate axes and whose classical partial derivative $\partial_k\tilde{u}$ agrees a.e.\ in $\Omega$ with the
distributional derivative $D_k u$ for all $k\in\{1,\dots,d\}$; see~\cite[Theorem~1.1.3/1]{Maz2011}
or~\cite[Theorem~2.1.4]{Ziemer1989}.
Let $g\colon\Omega\to\RR$ be a representative of $D_1 u$.
So by Fubini and the above, one has for $\Hm$-a.e.\ $x'\in\{e_1\}^\perp$ and $L := x'+\linspan\{e_1\}$ the following
three properties: firstly, $\partial_1\tilde{u}=g$ $\Hm[1]$-a.e.\ on $L\cap\Omega$ and 
\[
    \int_{L\cap\Omega}\abs{\partial_1\tilde{u}(t,x')}^p\dx[t]=\int_{L\cap\Omega}\abs{g(t,x')}^p\dx[t]<\infty;
\]
secondly,
$\tilde{u}$ is locally absolutely continuous in $L\cap\Omega$ and
for all $[a,b)\subset L\cap\Omega$ and $x\in[a,b)$ one has
\[
   \tilde{u}(x) = \tilde{u}(a) + \int_{a_1}^{x_1}\partial_1\tilde{u}(t,x')\dx[t];
\] 
thirdly, one has $u=\tilde{u}$ $\Hm[1]$-a.e.\ on $L\cap\Omega$ and therefore, due to continuity, $u=\tilde{u}$ everywhere on
$L\cap\Omega$.
Let $M$ be the set of the $x'\in\{e_1\}^\perp$ with the above properties.
Fix $x'\in M$ and $L$ as above and let $[a,b)\subset L\cap\Omega$ be bounded.
Then $b\in\clos{\Omega}$ and
\[  
    u(b) = \lim_{x\to b,\,x\in[a,b)} u(x)
	=\tilde{u}(a)+\lim_{x\to b}\int_{a_1}^{x_1}\partial_1\tilde{u}(t,x')\dx[t] =
u(a)+\int_{a_1}^{b_1} g(t,x')\dx[t].
\]
Hence $u$ is absolutely continuous on $[a,b]$.
\end{proof}
\begin{remark}
It is well-known that an element $u\in\Wone$ has a representative that is locally absolutely continuous on almost every
line parallel to the coordinate axes.
The central point of Lemma~\ref{lem:reg-lines} is to ensure that if $u\in\Wone\cap\Cbar$ then
the absolute continuity (along almost every line) extends up to the first point where the boundary is hit \emph{and} that
$u$ has the appropriate value at such a boundary point.
Of course, in general not every boundary point can be approached from $\Omega$ on a line segment, see Figures~\ref{fig:forest}
or~\ref{fig:wiggle}.

\begin{figure}
\centering
\includegraphics{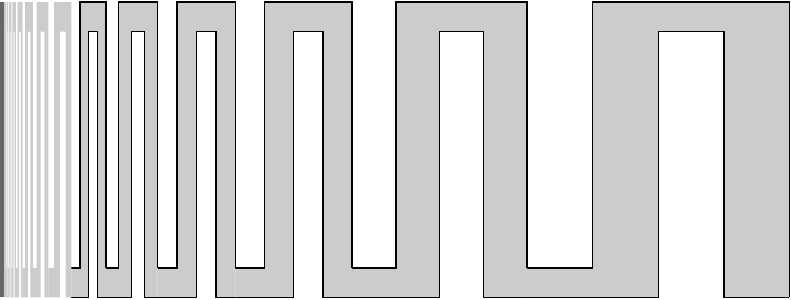}
\caption{A connected domain where the points in the dark grey line segment on the left are `inaccessible' from the inside with respect to the geodesic metric.}
\label{fig:wiggle}
\end{figure}

\end{remark}

The following proposition shows that if $\Omega$ has strictly positive upper density at a rectifiable part of the boundary then the approximative trace is unique there.
We shall see in Theorem~\ref{thm:sing-dens-0} that the rectifiability condition is superfluous.
\begin{proposition}\label{prop:trace-unique-rect}
Let $\Omega\subset\RR^d$ be open and $1\le p<\infty$.
Suppose that $A\subset\Gamma_\rect$ is Borel
and $\Hm$-a.e.\ $z\in A$ satisfies $\overline{D}(\Omega,z)>0$.
Then the approximative trace in $\Wone$ is unique on $A$.
\end{proposition}
\begin{proof}
Assume for contradiction that the approximative trace is nonunique on $A$.
So by Lemma~\ref{lem:nonuni-witness} there exist a compact set $K\subset A$ with $\Hm(K)>0$ and
a sequence $(u_n)$ in $\Wone\cap\Cbar$ such that $0\le u_n\le 1$, $u_n=1$ on $K$ and $u_n\to 0$ in $\Wone$.
Clearly $K$ is rectifiable since $K\subset\Gamma_\rect$.

Let $z_0\in K$ be a fixed boundary point with the properties specified in Lemma~\ref{lem:reg-rect}.
Additionally we may suppose that $\overline{D}(\Omega,z_0)>0$ and
$\Theta(\Gamma\setminus K,z_0)=0$; the latter is possible due to~\cite[Theorem~6.2]{Mat95} and since
$K\subset\Gamma_\loc$.

We use local orthogonal coordinates at $z_0$ such that the vector $e_d$ is a normal vector for $V$ and $z_0$ corresponds to $0$.
Denote by $C_r$ the open cube $(-r,r)^d$ in these local orthogonal coordinates.
Whenever we write $(x',t)$ for an element of $C_r$ in the following,
we consider this as coordinates in $V\times (-r,r)$.
Let $L_r := \{(x',t)\in C_r : t\in (-r,0)\}$ be the lower half of $C_r$ in the $e_d$ direction.
Choosing the sign of $e_d$ appropriately, we may assume that $\overline{D}(\Omega\cap L_r,z_0)>0$.

By choosing $\eps\in(0,\frac{1}{3})$ small enough, we may suppose that there exists a $\delta>0$ such that
\[
    \limsup_{r\to 0+}\frac{\meas{\Omega\cap L_r\setminus N_{\eps r}(V)}}{(2r)^d} > \delta.
\]
We may further decrease $\eps$ to ensure $0<\eps\le\frac{\delta}{3}$.
Fix such an $\eps$ and let $D$ and $f$ be as in Lemma~\ref{lem:reg-rect}\,\ref{en:rect-lipgraph}.
The assumptions regarding $z_0$ imply that there exists an $r_0>0$ such that
\begin{equation}\label{eq:dens-est}
    (1-\eps)(2r)^{d-1} \le \Hm(D\cap C_r)\le \Hm(\Gamma\cap C_r) \le (1+\eps)(2r)^{d-1}
\end{equation}
and
\[
    \Hm(\Gamma\cap C_r\setminus N_{\eps r}(V)) \le \eps(2r)^{d-1}
\]
for all $0<r\le r_0$.

Fix an $r\in (0,r_0]$ such that 
\[
    \meas{\Omega\cap L_r\setminus N_{\eps r}(V)} \ge \delta(2r)^d.
\]
In particular, from now on the cube $C_r$ is fixed.
Let $E := \Gamma\cap C_r\setminus \graph f$ be the `bad' boundary points in $C_r$.
We remove them from $C_r$ by setting $\tilde{D} := (D\setminus P_V E)\cap C_r$,
$\tilde{C} := \{(x',t)\in C_r : x'\in\tilde{D}\}$ and $\tilde{L} := \tilde{C}\cap L_r$, where $P_V$ denotes the orthogonal projection onto $V$.
Then by~\eqref{eq:dens-est} one has $\Hm(\tilde{D})\ge (1-3\eps)(2r)^{d-1}$ and $\meas{\tilde{L}}\ge \frac{1}{2}(1-3\eps)(2r)^d$.
Therefore
\begin{equation}\label{eq:lower-Om-est}
    \meas{\Omega\cap\tilde{L}\setminus N_{\eps r}(V)}\ge \delta(2r)^d-\frac{3}{2}\eps(2r)^d \ge\frac{\delta}{2}(2r)^d.
\end{equation}
Now let 
\[
    I := \{x'\in\tilde{D} : \text{there exists a $t$ such that $(x',t)\in\Omega\cap\tilde{L}\setminus N_{\eps r}(V)$}\}.
\]
Observe that for all $x'\in I$ the whole open line segment from $(x',-r)$ to $(x',f(x'))$ is contained in $\Omega$.
In particular, for all $x'\in I$  one has that $f(x')$ is the unique $t\in(-r,r)$ such that $(x',t)\in \Gamma$.
The measure of $\Omega$ in $\tilde{L}\setminus N_{\eps r}(V)$ can be bounded by
\[
    \meas{\Omega\cap\tilde{L}\setminus N_{\eps r}(V)} \le\Hm(I)r.
\]
Hence by~\eqref{eq:lower-Om-est} it follows that
\[
    \Hm(I) \ge \delta(2r)^{d-1}>0.
\]

Let $\eta\in\Cinfc[\{(x',t):x'\in V, t\in(-r,\infty)\}]$ be such that $0\le\eta\le 1$ and $\eta=1$ on $N_{\eps r}(V)\cap C_r$.
In particular, $\eta u_n=1$ on $\graph f\cap C_r$ and $\eta u_n=0$ on $I-re_d$.
As $\eta u_n\in\Wone\cap\Cbar[\Omega]$, 
it follows from Lemma~\ref{lem:reg-lines} that $(\eta u_n)(x',\cdot)$ is absolutely continuous on $[-r,f(x')]$ for $\Hm$-a.e.\ $x'\in I$.
Thus by the fundamental theorem of calculus, Fubini's theorem, Lemma~\ref{lem:reg-lines} and H\"older's inequality one obtains
\[
    0<\Hm(I)=\int_I\int_{-r}^{f(x')} \partial_d(\eta u_n)(x',t) \dx[t]\dx[x'] \le C\norm{u_n}_{1,p;C_r\cap\Omega}\le C\norm{u_n}_{1,p;\Omega}\to 0
\]
as $n\to\infty$, which is a contradiction.
\end{proof}
\begin{remark}
In the proof of Proposition~\ref{prop:trace-unique-rect} we 
control the behaviour locally on the boundary by the derivative in the interior.
This is reminiscent of the Gauss--Green divergence theorem.
In fact, there is a suitable version of the Gauss--Green theorem that can be applied in our setting,
see~\cite[Theorem~6.4]{Fed45} and~\cite[Remark~4.8]{Fed46}.
The resulting proof, however, would be less self-contained and still require analogous measure geometric considerations. 
\end{remark}

While the following proposition is a direct consequence of important classical results about the Lebesgue area and continuous functions of bounded variation, it appears that its statement is not very well-known.
It uses the main result of~\cite{Fed1960:np-surf}, which is a higher-dimensional version of the fundamental fact that a curve is rectifiable if and only if its one-dimensional Hausdorff measure is finite.
\begin{proposition}\label{prop:cont-graph-nice}
Let $I\subset\RR^{d-1}$ be a nontrivial bounded open rectangle and $g\colon \clos{I}\to\RR$ continuous.
Define $E := \{ (x,t)\in I\times\RR : t<g(x)  \}\subset\RR^d$.
Suppose that $\Hm(\graph g)<\infty$. Then $\Hm(\graph g\setminus\mbdy E)=0$.
In particular, $\graph g$ is rectifiable and $D(E,z)=\frac{1}{2}$ for $\Hm$-a.e.\ $z\in\graph g$.
\end{proposition}
\begin{proof}
It follows from~\cite{Fed1960:np-surf} that $\graph g$ is rectifiable and that $\alpha(g)=\Hm(\graph g)$,
where $\alpha(g)$ denotes the Lebesgue area of $g$.
By~\cite[Theorem~9.1]{Kri57} one has $g\in\BV(I)$ and $\alpha(g)=V(\mu,I)$, where 
$\mu$ is the $\RR^d$-valued vector measure $(\Hm[d-1]\niv I,D_1g,\dots, D_{d-1}g)$
and $V(\mu,I)$ denotes the total variation of $\mu$ on $I$.
Since $V(\mu,I)=P(E, I\times\RR)$ by~\cite[Theorem~1.10]{Mir64} and $P(E, I\times\RR)=\Hm(\graph g\cap\mbdy E)$ by~\cite[(3.62)]{AFP00}, the claim follows.
The statement about the density follows from~\cite[Theorem~3.61]{AFP00}.
\end{proof}

We now show that in the case of continuous boundary $\Hm$-a.e.\ point in 
$\Gamma_\loc$ belongs to the measure theoretic boundary of $\Omega$, which implies both rectifiability and the density condition in Proposition~\ref{prop:trace-unique-rect} on $\Gamma_\loc$.
Thanks to the beautiful results employed in the proof of Proposition~\ref{prop:cont-graph-nice}, we can now avoid technical arguments.
\begin{corollary}\label{cor:cont-mbdy}
Suppose $\Omega\subset\RR^d$ open has continuous boundary. Then $\Hm(\Gamma_\loc\setminus \mbdy\Omega)=0$.
In particular, $\overline{D}(\Omega,z)>0$ for $\Hm$-a.e.\ $z\in\Gamma_\loc$.
\end{corollary}
\begin{proof}
As $\Gamma_\loc$ is relatively open in $\Gamma$ and $\sigma$-compact, and membership of a boundary point in $\mbdy\Omega$ is a local
measure geometric property, it suffices to consider the question locally,
where $\Omega$ can be written as the subgraph of a continuous function.
We may assume that in local orthogonal coordinates there exists a rectangle 
$R:=I\times(-H,H)$, where $I\subset\RR^{d-1}$ is an open cube and $H>0$,
and a continuous function $g\colon I\to(-\frac{H}{2},\frac{H}{2})$ that has a continuous extension to $\clos{I}$ such that $\Hm(\graph g)<\infty$, $R\cap\Omega=\{(x',t)\in R: t<g(x')\}$ and
$R\cap\Gamma=\{(x',g(x')):x'\in I\}=\graph g$. 
Let $E:=\{(x',t)\in I\times\RR :t<g(x')\}$. One has $D(\Omega,z)=D(R\cap\Omega,z)=D(E,z)$ for all $z\in\graph g$.
By Proposition~\ref{prop:cont-graph-nice} one has $D(\Omega,z)=\frac{1}{2}$ for $\Hm$-a.e.\ $z\in\graph g$.
\end{proof}

Combining this with Proposition~\ref{prop:trace-unique-rect}, we obtain the following remarkable result.
\begin{theorem}\label{thm:cont-unique}
Suppose $\Omega\subset\RR^d$ open has continuous boundary and $1\le p\le\infty$. Then the approximative trace in $\Wone$ is unique.
\end{theorem}
\begin{proof}
By Corollary~\ref{cor:cont-mbdy} one has $\Hm(\Gamma_\loc\setminus\mbdy\Omega)=0$ and therefore $\overline{D}(\Omega,z)>0$ for $\Hm$-a.e.\ $z\in\Gamma_\loc$. Moreover, it follows from~\cite[Theorems~3.59 and~3.61]{AFP00} that $\Gamma_\loc$ is rectifiable.
So the assumptions of Proposition~\ref{prop:trace-unique-rect} are satisfied for $A=\Gamma_\loc$, which proves the claim.
\end{proof}
\begin{remark}
Theorem~\ref{thm:cont-unique} answers the corresponding 
open question in~\cite[Section~9]{AtE2011:DtN}.

This yields yet another remarkable property of domains with continuous boundary.
Other remarkable properties are the density property in Proposition~\ref{prop:cont-smooth-dense}, the related equality $\Wonez=\Wbarz$ resulting in the stability of the Dirichlet problem, see the proof of~\cite[Theorem~V.4.7]{EE87}, \cite{AD08} and~\cite{Hed2000:stab}, or the compact embedding of $\Wone$ into $\Lp$ if $\Omega$ is in addition bounded, see~\cite[Theorem~V.4.17]{EE87}.
The property that an open set has continuous boundary can be characterised via the segment property~\cite[Theorem~V.4.4]{EE87}.
\end{remark}

We now show that on parts of $\Gamma$ where $\Omega$ has density $1$ the approximative trace is always unique.
It is important that in contrast to Proposition~\ref{prop:trace-unique-rect} no rectifiability is assumed here. 
\begin{proposition}\label{prop:dens1-uniq}
Let $\Omega\subset\RR^d$ be open and $A\subset\Gamma_\loc$ Borel.
Suppose that $D(\Omega,z)=1$ for $\Hm$-a.e.\ $z\in A$. 
Then the approximative trace is unique on $A$.
\end{proposition}
\begin{proof}
We use notation and results from~\cite{FZ72}.
Let $K\subset A\cap\Gamma_\sing^p$ be compact and $\eps>0$.
Choose a $\delta>0$ sufficiently small such that $\Hm(U\cap \Gamma)\le \Hm(K)+\eps$,
where $U := \{ x\in\RR^d : \dist(x,K)<\delta \}$.
Note that $U$ is open in $\RR^d$ and $K\subset U$.

By Lemma~\ref{lem:nonuni-witness} there exists a sequence $(u_n)$ in $\Wone\cap\Cbar$ such that $0\le u_n\le 1$, $u_n=1$ on $K$, $\lim_{n\to\infty}u_n = 0$ in $\Wone$ and $\supp u_n\Subset U$.
After passing to a subsequence, we may assume that $\int_\Omega\abs{\nabla u_n}\le \frac{1}{n}$ for all $n\in \NN$.
Let $n$ be fixed.
By the coarea formula there exists a $t=t(n)\in (0,1)$ such that
\[
    P([u_n>t],\Omega)\le \int_\Omega\abs{\nabla u_n } \le \frac{1}{n},
\]
where $[u_n>t] = \{x\in\Omega : u_n(x)>t\}$.
Observe that $\mbdy[u_n>t]\subset(\Omega\cup \Gamma)\cap U$. Moreover, by assumption 
$D([u_n>t],z)=1$ for $\Hm$-a.e.\ $z\in K$ since $u_n>t$ on a neighbourhood of $K$ in $\clos{\Omega}$.
Hence $\Hm(K\cap\mbdy[u_n>t])=0$.
Therefore
\begin{align*}
    P([u_n>t],\RR^d) &= \Hm(\mbdy[u_n>t]) \\
        & = \Hm(\Omega\cap\mbdy[u_n>t]) + \Hm(U\cap\Gamma\cap\mbdy[u_n>t]\setminus K) \\
        & \le P([u_n>t],\Omega) + \Hm(U\cap\Gamma\setminus K) \\
        & \le \frac{1}{n}+\eps.
\end{align*}
Then $w := \one_{[u_n>t]}$ is an element of $\BV(\RR^d)$, $D([w\ge 1],z)=1$ for $\Hm$-a.e.\ $z\in K$ and
\[
    \norm{D w}(\RR^d) = P([w\ge 1],\RR^d) = P([u_n>t],\RR^d) \le\frac{1}{n}+\eps.
\]
Since such a function $w$ can be found for all $n\in\NN$ and $\eps>0$,
it follows that $\alpha=\alpha(K)=0$, where $\alpha$ is as in~\cite[Proposition on p.\,145]{FZ72}.
As $\alpha(K)=\Gamma_1(K)$, where $\Gamma_1$ denotes the $1$-capacity in $\RR^d$ as defined in~\cite{FZ72},
one obtains $\Gamma_1(K)=0$. By~\cite[final Proposition in Section~4]{FZ72} it follows that $\Hm(K)=0$.
\end{proof}
\begin{remark}
The proof of the identity $\alpha=\Gamma_1$ in~\cite{FZ72} uses refined machinery from the theory of currents.
In the special case considered here a technically simpler proof of Proposition~\ref{prop:dens1-uniq} can be given along the lines of~\cite[Theorem~5.6.3]{EG92}.
\end{remark}

By combining our previous results we now obtain the following measure geometric criterion for the uniqueness of the approximative trace.
A connection between $\Omega$ having Lebesgue density zero and the nonuniqueness phenomenon was suggested
in~\cite[p.\,85]{BD2010:alt-fk} and~\cite[p.\,941]{BG2010}.

\begin{theorem}\label{thm:sing-dens-0}
Let $\Omega\subset\RR^d$ be open and $1\le p\le\infty$. Then $D(\Omega,z)=0$ for $\Hm$-a.e.\ $z\in\Gamma_\sing^p$.
In particular, if $A\subset\Gamma_\loc$ is Borel such that $\overline{D}(\Omega,z)>0$ for $\Hm$-a.e.\ $z\in A$,
then the approximative trace is unique on $A$.
\end{theorem}
\begin{proof}
Suppose that $K\subset\Gamma_\sing^p$ is compact. Recall that $\Gamma_\sing^p\subset\Gamma_\loc$.
Similarly as in the proof of Proposition~\ref{prop:dens1-uniq} we can consider $U\cap\Omega$ instead of $\Omega$ with $K\Subset U$ and $\clos{U}\cap\Gamma\subset\Gamma_\loc$. We may suppose that $U$ is smooth and bounded.
Clearly $K\subset\partial(U\cap\Omega)$ and the density properties of $U\cap\Omega$ and $\Omega$ are identical at all points in $K$.

Furthermore, $U\cap\Omega$ has finite perimeter. 
So it follows from Proposition~\ref{prop:trace-unique-rect} that the approximative trace is 
unique on $K\cap\mbdy(U\cap\Omega)$ as $\mbdy(U\cap\Omega)$ is rectifiable
and $\overline{D}(U\cap\Omega,z)>0$ for all $z\in\mbdy(U\cap\Omega)$.
By Proposition~\ref{prop:dens1-uniq} the approximative trace is unique on $K\cap(U\cap\Omega)^1$.
Together this implies that
\[
    \Hm\bigl(K\cap\bigl\{z\in\RR^d: D(U\cap\Omega,z)=1\vee(\underline{D}(U\cap\Omega,z)<1 \wedge \overline{D}(U\cap\Omega,z)>0)\bigr\}\bigr)=0.
\]
In other words, $D(\Omega,z)=D(U\cap\Omega,z)=0$ for $\Hm$-a.e.\ $z\in K$.
\end{proof}
\begin{remark}
We shall see in Example~\ref{ex:p-dep} that it is not sufficient for the nonuniqueness of the approximative trace that
$\Omega$ has density $0$ on a substantial part of $\Gamma_\loc$. In fact, in general the set $\Gamma_\sing^p$ does depend on $p$.
\end{remark}

The following corollary is immediate.
\begin{corollary}
Let $\Omega\subset\RR^d$ be open. 
Then the approximative trace in $\Wone[p][\RR^d\setminus\clos{\Omega}]$ is unique on $\Gamma_\sing^p\cap\partial(\RR^d\setminus\clos{\Omega})$, where $\Gamma_\sing^p$ is considered with respect to $\Omega$ as usual.
\end{corollary}

The next corollary is a generalisation of~\cite[Proposition~5.5]{AW03}.
It shows that if part of the boundary is locally finite,
then not all of it can belong to the singular part $\Gamma_\sing^p$. 
\begin{corollary}
Let $\Omega\subset\RR^d$ be open and $1\le p\le\infty$. If $\Hm(\Gamma_\loc)>0$, then
$\Hm(\Gamma_\loc\setminus\Gamma_\sing^p)>0$.
\end{corollary}
\begin{proof}
The case $d=1$ is easy. Hence we suppose $d\ge 2$.
Assume for contradiction that $\Hm(\Gamma_\loc)>0$ and $\Hm(\Gamma_\loc\setminus\Gamma_\sing^p)=0$.
Let $z\in\Gamma_\sing^p$ be such that $D(\Omega,z)=0$. By Theorem~\ref{thm:sing-dens-0}, $\Hm$-a.e.\ $z\in\Gamma_\loc$ has this property.
For $r>0$ sufficiently small one has 
\[
    \Hm(\mbdy\Omega\cap B(z,r))\le\Hm(\Gamma\cap B(z,r))<\infty.
\]
Moreover, Theorem~\ref{thm:sing-dens-0} implies that $\Hm(\mbdy\Omega\cap\Gamma_\sing^p)=0$ and therefore by assumption $\Hm(\mbdy\Omega\cap B(z,r))=0$. 
So by the relative isoperimetric inequality~\cite[(3.43)]{AFP00}, the formula~\cite[(3.62)]{AFP00} and Federer's characterisation of
sets of locally finite perimeter~\cite[Theorem~4.5.11]{Federer1969} 
there exists a $C>0$ such that
\[
    \min\{\meas{\Omega\cap B(z,r)},\meas{B(z,r)\setminus \Omega}\}^{(d-1)/d} \le C P(\Omega,B(z,r))=C\Hm(\mbdy\Omega\cap B(z,r))=0.
\]
As $z\in\Gamma$ and $\Omega$ is open, we must have $\meas{B(z,r)\setminus\Omega}=0$.
This implies $D(\Omega,z)=1$, which is a contradiction.
\end{proof}

\section{Uniqueness of the trace in two dimensions}\label{sec:uniq-2d}

In~\cite{AW03} the first example of a disconnected open set in $\RR^2$ with nonunique approximative trace was given.
This example was then used to construct a connected example in three dimensions. 
A simpler three-dimensional, connected example was given in~\cite[Section 3, last paragraph on p.\,941]{BG2010}.
In private communication with the author, Wolfgang Arendt and Tom ter Elst raised the question of whether there exists a connected domain in $\RR^2$ with nonunique approximative trace.
In this section we prove the surprisingly strong statement that the approximative trace for a two-dimensional connected domain is always unique, thereby answering the question in the negative.

The results in~\cite{ACMM2001} (see also~\cite[Theorem~3]{FF09:fin-perim-plane}) 
imply that a connected, bounded domain in $\RR^2$ with a boundary of finite one-dimensional Hausdorff measure has a very specific structure: it is, up to a set of Lebesgue measure $0$, the interior of a rectifiable Jordan curve, from which the interior of possibly countably many rectifiable Jordan curves have been removed.
The next proposition makes this precise. It follows directly from~\cite[Proposition~2, Theorem~4 and Corollary~1]{ACMM2001}.
\begin{proposition}\label{prop:acmm-jordan}
Let $\Omega\subset\RR^2$ be connected, bounded and open such that $\Hm[1](\partial\Omega)<\infty$.
Then there exist a countable set $J$ and rectifiable closed Jordan curves $\gamma$, $\gamma_j$ for all $j\in J$ such that $\mbdy\Omega\equiv\gamma\cup\bigcup_{j\in J}\gamma_j$ modulo $\Hm[1]$, $\inter\gamma_j\subset\inter\gamma$ for all $j\in J$, $\inter \gamma_j\cap\inter\gamma_k=\emptyset$ for all $j,k\in J$ with $j\ne k$ and
\[
	A := \inter\gamma\setminus\bigcup_{j\in J}\inter\gamma_j \equiv\Omega\text{ modulo $\Hm[2]$.}
\]
\end{proposition}

The following lemma is an ingredient in the proof of Proposition~\ref{prop:acmm-jordan}.
\begin{lemma}[{\cite[Lemma~4]{ACMM2001}}]\label{lem:J-curve-rect}
Let $\gamma$ be a rectifiable Jordan curve in $\RR^2$. Then $\Hm[1](\gamma)=P(\inter\gamma )=P(\ext \gamma )$. 
In particular, $\Hm[1]$-a.e.\ point on $\gamma$ is in $\mbdy(\inter\gamma)$.
\end{lemma}

Note that it follows immediately from Lemma~\ref{lem:J-curve-rect} and Proposition~\ref{prop:trace-unique-rect}
that the approximative trace in $\Wone$ is unique if $\Omega\subset\RR^2$ is the interior of a rectifiable Jordan curve.

\begin{theorem}\label{thm:dense-Gloc-2dim}
Let $\Omega\subset\RR^2$ be open and connected. Then for $\Hm[1]$-a.e.\ $z\in\Gamma_\loc$ one has $D(\Omega,z)\in\{\frac{1}{2},1\}$.
\end{theorem}
\begin{proof}
Due to Lemma~\ref{lem:dens-Gloc} it suffices to show that $\Hm[1](\Gamma_\loc^0)=0$.
Let $z_0\in \Gamma_\loc$. Choose $r>0$ sufficiently small such that $\Hm[1](B(z_0,2r)\cap\Gamma)<\infty$ 
and $\Omega\setminus\clos{B}(z_0,r)\ne\emptyset$.
Set 
\[
	\Omega' := B(z_0,2r)\cap\Omega\cup(B(z_0,2r)\setminus\clos{B}(z_0,r)).
\]
Then $\Omega'$ is bounded, connected and open with $\Hm[1](\partial\Omega')<\infty$ and 
the upper and lower densities of $\Omega$ and $\Omega'$ at $z$ are equal
for all $z\in B(z_0,r)\cap\Gamma_\loc$.
So we may assume without loss of generality that $\Omega$ satisfies the conditions of Proposition~\ref{prop:acmm-jordan} and we shall use the notation introduced there.
Note that the density properties of $A$ and $\Omega$ are the same at every $z\in\RR^2$ and therefore $\mbdy A=\mbdy\Omega$. Moreover, $D(\Omega,z)=\frac{1}{2}$ for $\Hm[1]$-a.e.\ $z\in\gamma\cup\bigcup_j \gamma_j$.

Let $z\in\partial\Omega\setminus(\mbdy\Omega\cup\gamma\cup\bigcup_j\gamma_j)$.
Then $z\in\ext\gamma$ is impossible, since otherwise there exists an $r>0$ such that $B(z,r)\cap\inter\gamma=\emptyset$ and therefore $0<\Hm[2](\Omega\cap B(z,r))=\Hm[2](A\cap B(z,r))=0$, which would be a contradiction.
Similarly one shows that $z\in\ext\gamma_j$ for all $j\in J$.
As $z\notin\mbdy A$, we obtain $D(A,z)\in\{0,1\}$.
If 
\[
	z\in\inter A=\inter\gamma\setminus\Bigl(\bigcup_j\clos{\inter\gamma_j}\mathbin{\dot{\cup}} \clos{\bigcup_j\gamma_j}\setminus\bigcup_j\gamma_j\Bigr),
\]
then $D(A,z)=1$. Hence it remains to consider 
\[
	z\in(\inter\gamma\setminus\mbdy A)\cap\Bigl(\clos{\bigcup_j\gamma_j}\setminus\bigcup_j\gamma_j\Bigr).
\]
By~\cite[Theorem~6.2]{Mat95} we may suppose in addition that 
\begin{equation}\label{eq:dens-gammaj-0}
	\Theta(\gamma\cup\bigcup_j\gamma_j,z)=0.
\end{equation}

Assume for contradiction that $D(A,z)=0$. This implies that $D(\bigcup_j\inter\gamma_j,z)=1$.
We may suppose $z=0$.
Let $\eps\in(0,\frac{1}{3})$. Then there exists an $r>0$ such that the cube $C=(-r,r)^2$ satisfies
$\Hm[2](C\cap\bigcup_j\inter\gamma_j)>(4-\eps)r^2$ and $\Hm[1](\clos{C}\cap\bigcup_j\gamma_j)<\eps r$.
Hence there exists an $N\in\NN$ such that
\begin{equation}\label{eq:dens-paradox-box}
	\Hm[2](C\cap\bigcup_{j=1}^N\inter\gamma_j)>(4-\eps)r^2\quad\text{and}\quad \Hm[1](\clos{C}\cap\bigcup_{j=1}^N\gamma_j)<\eps r.
\end{equation}
Set $K :=\clos{C}\cap\bigcup_{j=1}^N\gamma_j$ and let $P_1$ be the projection onto the first component. Then $\Hm[1](P_1K)<\eps r$ and for $t\in(-r,r)\setminus P_1K$ one either has 
\[
	\{t\}\times(-r,r)\subset\bigcup_{j=1}^N\inter\gamma_j\quad\text{or}\quad \{t\}\times(-r,r)\cap\bigcup_{j=1}^N\inter\gamma_j=\emptyset. 
\]
If one would have $\{t\}\times(-r,r)\cap\bigcup_{j=1}^N\inter\gamma_j=\emptyset$ for all $t\in(0,\frac{r}{2})\setminus P_1K$, then
\[
	(4-\eps)r^2<\Hm[2]\Bigl(C\cap\bigcup_{j=1}^N\inter\gamma_j\Bigr)\le 4r^2 - 2r(\tfrac{1}{2}-\eps)r,
\]
 which is impossible.
Hence there exists a $t_1\in(0,\frac{r}{2})$ such that $\{t_1\}\times(-r,r)\subset \bigcup_{j=1}^N\inter\gamma_j$. We may suppose that $\{t_1\}\times(-r,r)\subset \inter\gamma_1$.
Let $t_0:=\sup\{t\in[0,t_1] : z=(0,t)\in\gamma_1\}$. Then $t_0\in(0,t_1)$ and $z_0=(t_0,0)\in\gamma_1$. Moreover, as $\{t_1\}\times(-r,r)\subset \inter\gamma_1$, there is a $z_1\in\gamma_1\cap\partial C$. Clearly $\dist(z_1,z_2)>\frac{r}{2}$ and thus $\Hm[1](\clos{C}\cap\gamma_1)>\frac{r}{2}$ by~\cite[Lemma~3.4]{Fal86}, which contradicts~\eqref{eq:dens-paradox-box}.
Consequently the assumption was incorrect and we instead have $D(A,z)=1$.
We have shown that $D(\Omega,z)\in\{\frac{1}{2},1\}$ for $\Hm[1]$-a.e.\ $z\in\partial\Omega$.
\end{proof}

\begin{figure}
\centering
\includegraphics{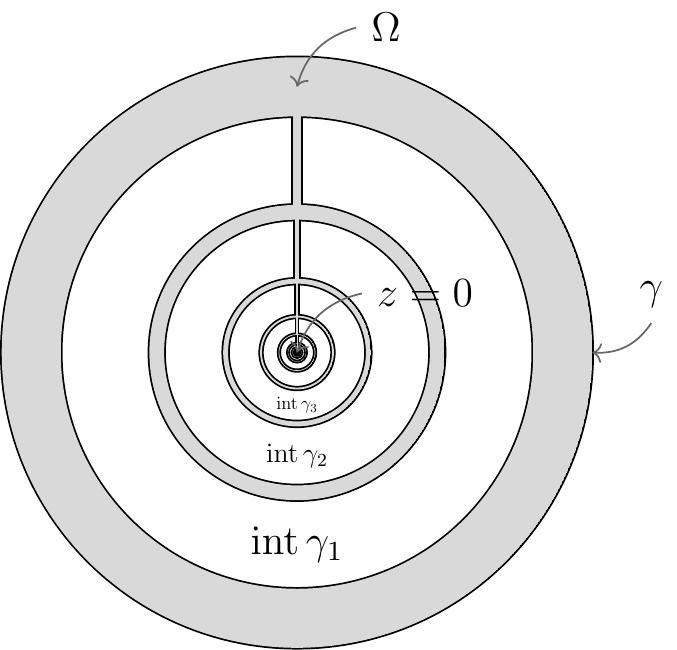}
\caption{A connected domain $\Omega\subset\RR^2$ such that, for a suitable choice of parameters, $\Hm[1](\partial\Omega)<\infty$, $z=0$ is in~$\partial\Omega\cap\clos{\bigcup_j\gamma_j}\setminus\bigcup_j\gamma_j$ and $D(\Omega,z)=0$.}
\label{fig:con-dens-0}
\end{figure}
The domain in Figure~\ref{fig:con-dens-0} shows that $\Omega$ can have density $0$ at a limit point of $\bigcup_j\gamma_j$ in $\inter\gamma$, but then~\eqref{eq:dens-gammaj-0} cannot be satisfied at that point.

The following is now an immediate consequence of Theorems~\ref{thm:dense-Gloc-2dim} and~\ref{thm:sing-dens-0}.
\begin{corollary}\label{cor:uniq-two-dim}
Let $1\le p\le\infty$ and $\Omega\subset\RR^2$ be open and connected. 
Then the approximative trace in $\Wone$ is unique.
\end{corollary}

So in~\cite[Examples~4.2 and~4.3]{AW03} or~\cite[Section 3, last paragraph on p.\,941]{BG2010} it was essential to work in (at
least) three dimensions in order to be able to construct a connected domain with nonunique approximative trace.
Moreover, Theorem~\ref{thm:dense-Gloc-2dim} is of independent interest since it applies in a wider context.
For example, also in~\cite[Example~1]{BK2010} three dimensions are necessary for a connected example.

\section{Examples and applications}\label{sec:ex-app}

It is easy to see that for $d\ge 2$ there are examples of open sets $\Omega\subset\RR^d$ such that
the approximative trace is not unique in $\Wone$ for all $p\in[1,\infty)$ with $\Gamma_\sing^p$ being independent of $p$.
For $d\ge 3$ one can in addition require that $\Omega$ is connected. 
In fact, it suffices to consider an $\Omega$ as in Figure~\ref{fig:forest} with radii that decrease suitably quickly or a domain as in~\cite[Example~4.4]{AtE2011:DtN}.

We next present an example of a quasi-convex, connected, open set $\Omega\subset\RR^3$ where 
the uniqueness of the approximative trace in $\Wone$ depends on the value of $p$.
More precisely, in this example there exists a $p_0\in (1,\infty)$ such that the approximative trace is unique for $p>p_0$ and not unique for $p<p_0$.
We need the following notation and proposition from~\cite{BS2001}.
Let $\Omega\subsetneq\RR^d$ be open and connected. Let $\alpha\in(0,1]$. For all $x,y\in\Omega$ one defines
\[
    d_{\alpha,\Omega}(x,y) := \inf_\gamma \int_\gamma \bigl(\dist(z,\partial\Omega)\bigr)^{\alpha-1}\dx[z],
\]
where the infimum is taken over all rectifiable curves in $\Omega$ connecting $x$ and $y$.
This defines a metric with respect to which elements of $\Wone$ are uniformly continuous for a suitable $p=p(\alpha)>d$.
The case $\alpha=1$ corresponds to $p=\infty$, in which case this boils down to the well-known fact that elements in $\Wone[\infty]$ are uniformly continuous with respect to the geodesic distance in $\Omega$.

\begin{proposition}[cf.~{\cite[Theorem~3.2]{BS2001}}]\label{prop:buck-stan}
Let $\Omega\subsetneq\RR^d$ be open and connected. Let $p>d$ and $\alpha := \frac{p-d}{p-1}\in (0,1)$. Then there exists a constant $C=C(d,p)$ such that for all $u\in\Wone$ and $x,y\in\Omega$ one has
\begin{equation}\label{eq:buck-stan}
    \abs{u(x)-u(y)}\le C \bigl(d_{\alpha,\Omega}(x,y)+\abs{x-y}^\alpha\bigr)^{(p-1)/p}\norm{\nabla u}_{L^p(\Omega)}.
\end{equation}
\end{proposition}

\begin{figure}
\centering
\includegraphics{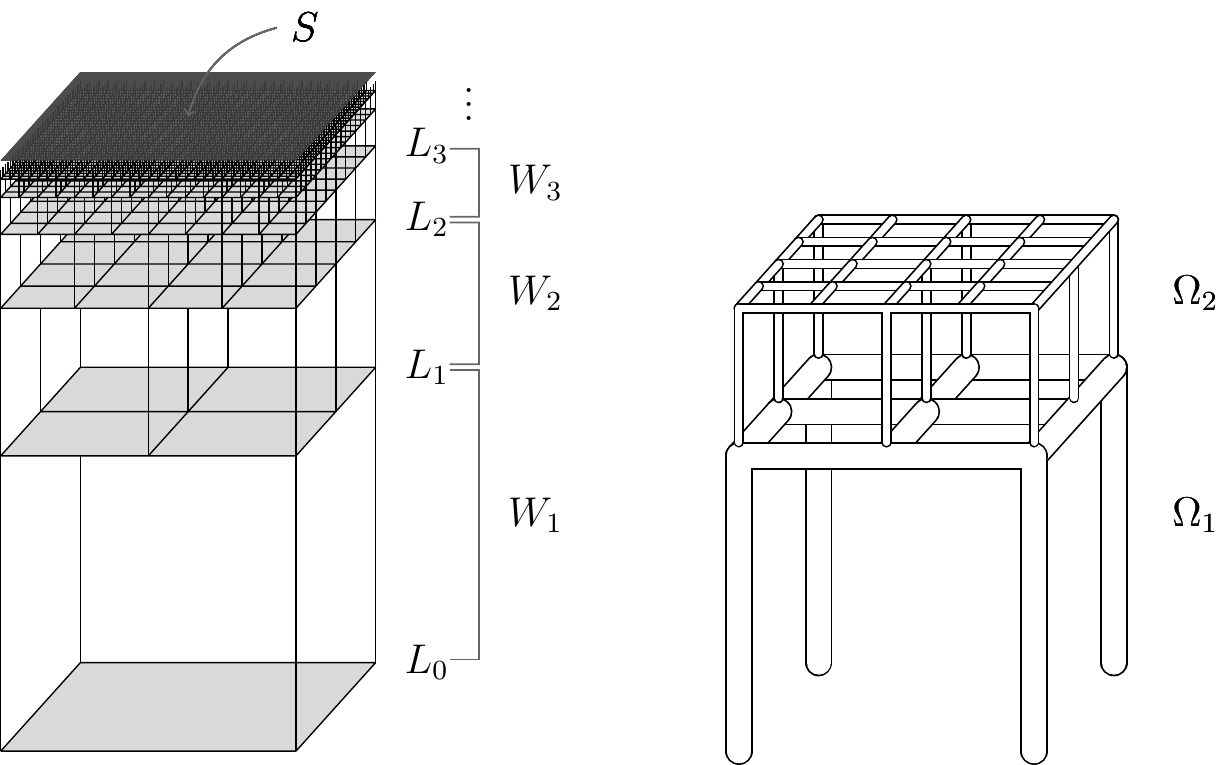}
\caption{On the left the wireframe structure $W$ is depicted, while on the right the union of $\Omega_1=W_1+B(0,r_1)$ and $\Omega_2=W_2+B(0,r_2)$ is shown.}
\label{fig:p-dep-wireframe}
\end{figure}

\begin{example}\label{ex:p-dep}
In order to construct the desired domain $\Omega$ we start by describing an auxiliary wireframe structure $W$ in $\RR^3$
that is depicted in Figure~\ref{fig:p-dep-wireframe}.
Informally,
the wireframe structure $W$ is composed of refining dyadic grids in countably many layers stacked on top of each other with connections from the grid points of one layer to the corresponding grid points of the following layer.

We think of the $z$ direction as pointing upward
and start with the edges of an axis-aligned unit cube in $\RR^3$. The plane containing the bottom face of the unit cube (say $z=0$) we consider as the zero-th layer,
the plane containing the top face ($z=1$) as the first layer. We add to the wireframe the edges of four cubes with side length $\frac{1}{2}$ that rest in a regular axis-aligned grid on the top face of the original unit cube. The plane containing the top faces of these four cubes is the second layer.
In the next step we add the edges of $16$ cubes with side length $\frac{1}{4}$ resting on the second layer.
We continue this construction to obtain the wireframe structure $W\subset\RR^3$ as the union of all the arising edges.
Note that the height of the structure is $\sum_{k=0}^\infty2^{-k}=2$ and that $\clos{W}$ contains the closed unit square $S$ above the wireframe at height $2$.
The intersection of $W$ with the $N$th layer is a regular grid composed of $4^{N}$ squares with side length $2^{-N}$.
The grid on layer $N-1$ is connected to that on layer $N$ by $(2^{N-1}+1)^2$ vertical lines of length $2^{-(N-1)}$ connecting the corresponding grid points. The wireframe structure $W$ is depicted in Figure~\ref{fig:p-dep-wireframe}.

We define $L_N := \sum_{k=0}^{N-1} 2^{-k}=2(1-2^{-N})$ for all $N\in\NN_0$ and
set $W_N := \{ p=(x,y,z)\in W : L_{N-1}<z\le L_N\}$
for all $N\in\NN$.

Next we construct $\Omega\subset\RR^3$ by suitably `blowing up' the set $W\cap\{(x,y,z):z> 0\}$.
Let a sequence $(r_k)_{k\in\NN}$ in $(0,\infty)$ be given that decreases sufficiently quickly.
A suitable choice for $(r_k)$ will be specified later.
We set $\Omega_N := W_N + B(0,r_N)$ and $\Omega := \bigcup_{N\in\NN} \Omega_N$.
Clearly $\Omega_N$ and $\Omega$ are connected open sets. It can readily be verified that $\Omega$ is quasi-convex,
 i.e.\ the intrinsic geodesic distance in $\Omega$ is comparable with the Euclidean distance.

As $\Hm[1](W_N)= 2(2^N+1) + 2^{-(N-1)}(2^{N-1}+1)^2$ (which grows like $2^N$), we obtain for a suitable $C>0$ the estimate
\[
    \Hm[2](\partial\Omega) \le \Hm[2](S) + C \sum_{k=1}^\infty 2^k r_k.
\]
To ensure that $\Hm[2](\partial\Omega)<\infty$, we require 
\begin{equation}\label{eq:rk-finite-Hm}
    \sum_{k=1}^\infty 2^k r_k<\infty.
\end{equation}
The volume of $\Omega_N$ is of the order $2^N r_N^2$. 
Obviously~\eqref{eq:rk-finite-Hm} implies that $\meas{\Omega_N}\to 0$ as $N\to\infty$. Moreover,
it is straightforward to check that $\Omega$ has density $0$ at $S$.

For all $N\in\NN$ let $u_N\in \Wone\cap\Cbar$ be such that
\[
    u_N(x,y,z) = \begin{cases}
    0 & \text{for $z\le L_{N-1}+\frac{1}{4} 2^{-(N-1)}$,} \\
    1 & \text{for $z\ge L_{N-1}+\frac{3}{4} 2^{-(N-1)}$,} \\
    2^N(z-L_{N-1}-\frac{1}{4}2^{-(N-1)}) & \text{otherwise.}
    \end{cases}
\]
Then
\[
    \int_\Omega \abs{D_z u_N}^p = (2^{N-1}+1)^2\int_0^{2^{-N}} 2^{pN} \pi r_N^2\dx[t] \sim 2^{(p+1)N}r_N^2.
\]
If $p\ge 1$ and $(r_k)$ are such that 
\begin{equation}\label{eq:rk-nonunique}
    2^{(p+1)N}r_N^2\to 0
\end{equation} 
as $N\to 0$, it follows that
$S\subset \Gamma_\sing^p$ (up to an $\Hm$-nullset). In particular, in this case the approximative trace is not unique in $\Wone$.

We now show that the approximative trace is unique for suitably chosen $(r_k)$ provided $p>3$ is sufficiently large.
First note that $\Omega$ has strictly positive density
everywhere on the rectifiable (and, if fact, locally Lipschitz) boundary $\partial\Omega\setminus S$. 
So by Proposition~\ref{prop:trace-unique-rect} it suffices to control the approximative trace on $S$.
Fix $z\in S$.
For all $N\in\NN$ let $x_N$ be one of the grid points in the $N$th layer of $W$ with minimal Euclidean distance to $z$.
Clearly $\abs{x_N-z}\to 0$ as $N\to\infty$. 
We show that $(x_N)$ is Cauchy with respect to $d_{\alpha,\Omega}$ for $\alpha := \frac{p-3}{p-1}$ provided
\begin{equation}\label{eq:rk-dist-bound}
    \sum_{k=1}^\infty 2^{-k}r_k^{\alpha-1}<\infty.
\end{equation}
Observe that
\[
    d_{\alpha,\Omega}(x_{N-1},x_N)
    \le 2\int_0^{2^{-(N-1)}}r_N^{\alpha-1} \sim 2^{-N}r_N^{-2/(p-1)}
\]
by taking the obvious Euclidean geodesic from $x_{N-1}$ to $x_N$ in $W$.
So the distances $d_{\alpha,\Omega}(x_{N-1},x_N)$ are summable over $N$ if~\eqref{eq:rk-dist-bound} holds, in which case, by estimating telescopic sums, it follows that $(x_N)$ is Cauchy with respect to $d_{\alpha,\Omega}$. 

Suppose now that $(x_N)$ is Cauchy with respect to $d_{\alpha,\Omega}$ and that $(u_n)$ is a sequence in $\Wone\cap\Cbar$ such that $u_n\to 0$ in $\Wone$. As $p>d=3$, one has $u_n(x)\to0$ for all $x\in \Omega$.
For the uniqueness
of the approximative trace in $\Wone$, it suffices to show that $\abs{u_n(z)}\to 0$ as $n\to\infty$.
One has
\begin{align*}
    \abs{u_n(z)-0} &= \lim_{k\to\infty}\abs{u_n(x_k)} \\
        &\le \lim_{k\to\infty}\abs{u_n(x_k)-u_n(x_l)} + \abs{u_n(x_l)} \\
        &\le \limsup_{k\to\infty}C\bigl(d_{\alpha,\Omega}(x_k,x_l)+\abs{x_k-x_l}^\alpha\bigr)^{(p-1)/p} + \abs{u_n(x_l)},
\end{align*}
where we used~\eqref{eq:buck-stan} in the last step and the constant $C>0$ does not depend on $n$, $k$ and $l$.
Let $\eps>0$. As $(x_N)$ is Cauchy with respect to $d_{\alpha,\Omega}$ and convergent with respect to the Euclidean
distance, we can choose $n_0\in\NN$ such that the term with the limes superior is less than $\eps$ for all $l\ge n_0$.
Because $u_n(x_l)\to 0$ as $n\to\infty$, we may assume that $\abs{u_n(x_l)}<\eps$ for all $n\ge n_0$.
This proves that $u_n(z)\to 0$ as $n\to\infty$ for all $z\in S$.

Finally, suppose that $r_k=2^{-ck}\frac{1}{k^2}$. Then for~\eqref{eq:rk-finite-Hm} it suffices to require $c\ge 1$.
By~\eqref{eq:rk-nonunique} it follows that the approximative trace is not unique if $p\le 2c-1$.
Moreover, by~\eqref{eq:rk-dist-bound} and the above, the approximative trace is unique if $p>2c+1$.
In combination with Lemma~\ref{lem:Gsing-incl}, 
this shows that there exists a $p_0\in[2c-1,2c+1]$ such that the approximative trace is unique in $\Wone$
for $p>p_0$ and not unique for $p<p_0$. Note that if one chooses $c>2$ then one certainly has $p_0>3$.
\end{example}

\begin{remark}\leavevmode
\begin{asparaenum}[(a)]
\item
It is clear that Example~\ref{ex:p-dep} can be adapted to the case $d\ge 3$.
However, we neither know the exact value of $p_0$ in Example~\ref{ex:p-dep}, nor what happens at $p_0$.
We showed that $\Gamma_\sing^p=\emptyset$ if $p\in(p_0,\infty]$ and $\Gamma_\sing^p=S$ if $p\in[1,p_0-2)$ (up to $\Hm[2]$-nullsets), but we did not prove that $\Gamma_\sing^p\in\{\emptyset,S\}$ for all $p\in[1,\infty]$, even though this seems likely.

\item
It would be interesting to be able to pinpoint any given $p_0$, in particular for the case $p_0\in[1,d]$.
Observe that our arguments only locate $p_0$ within some interval.
In particular, for $d=3$ we cannot guarantee that $p_0\in[1,3)$ 
since we need $p>3$ for Proposition~\ref{prop:buck-stan}.

\item
While it is easily seen that in dimension one the approximative trace is always unique,
it might be interesting to investigate whether the uniqueness of the approximative trace is $p$-dependent in two dimensions.
Note that for our argument in Example~\ref{ex:p-dep} it was essential that $\Omega$ was connected.
However, Corollary~\ref{cor:uniq-two-dim} implies that in two dimensions the approximative trace is always unique if $\Omega\subset\RR^2$ is connected.

\item
The example shows that, at least for $p$ sufficiently large, the approximative trace can be unique in $\Wone$, even though $\Omega$ has density $0$ on a part of $\Gamma_\loc$ with positive Hausdorff measure.
We point out that $\Omega$ can have density zero at $\Gamma_\nloc$ without leading to nonuniqueness of the approximative trace, see Figure~\ref{fig:wiggle} or a suitably modified version of~\cite[Example~2.5.5]{Bie2005:thesis}.
\end{asparaenum}
\end{remark}

Thanks to Example~\ref{ex:p-dep} we can formulate the next proposition.
\begin{proposition}
Let $\Omega\subset\RR^d$ be open. Then in general the relatively $p$-polar subsets of $\partial\Omega$ modulo $\Hm$ depend on $p\in(1,\infty)$. In particular, it is possible that $\partial\Omega$ is $p$-admissible for some $p\in(1,\infty)$, but not for others.
\end{proposition}

We next use Example~\ref{ex:p-dep} to address the following question, which was raised for $p=2$ in the last paragraph of~\cite[Section~4]{AtE2011:DtN}:
If $\Hm(\partial\Omega)<\infty$ and every element $u\in\Wt$ has an approximative trace in $Y=\Lp[p][\partial\Omega]$, is then the approximative trace unique on $\Omega$? Here we ask for convergence in $\Lp[p][\partial\Omega]$ on the boundary.

\begin{example}\label{ex:p-conj-ate}
We use the notation from Example~\ref{ex:p-dep}. We suppose that $p>3$ and $p\le 2c-1$. In particular, $c>2$. So the approximative trace in $\Loneloc$ is not unique and $\Gamma_\sing^p=S$. The sequence $(u_N)$ constructed in Example~\ref{ex:p-dep} also shows that the approximative trace in $\Lp[q][\partial\Omega]$ is not unique for all $q\in[1,\infty)$.

We shall show that for suitable $c$, $p$ and $q$ every element of $\Wone$ has an approximative trace in $\Lp[q][\partial\Omega]$.
Observe that every $u\in\Wone$ has a continuous representative in $C(\Omega)$ that has a (unique) continuous extension to $\partial\Omega\setminus S$ since $\Omega$ has Lipschitz boundary locally at every point of $\partial\Omega\setminus S$.
Fix $u\in\Wone$.
By considering $u(1-u_N)$ instead of $u$, we only need to show that $u$ has an approximative trace in $\Lp[q][\partial\Omega\setminus S]$.
Hence it suffices to show that the continuous extension of $u$ to $\partial\Omega\setminus S$ (in the following also denoted by $u$) is in $\Lp[q][\partial\Omega\setminus S]$.

Let $S_N := \partial\Omega_N$. Then $\Hm[2](S_N)$ is of the order $2^N r_N$.
Fix $N\in\NN$ with $N\ge 2$ and $z\in S_N$. We want an estimate of $\abs{u(z)}$ based on the values of $u$ in $\Omega_1$.
Let $w\in W\cap\clos{\Omega_N}$ such that $\abs{z-w}$ is minimal. Let $y$ be a grid point in the $N$th layer of $W$ such that $\abs{w-y}$ is minimal. Let $x$ be a grid point in the $0$th layer of $W$ such that $\abs{y-x}$ is minimal.
Now
\begin{align*}
    d_{\alpha,\Omega}(x,z) &\le d_{\alpha,\Omega}(x,y) + d_{\alpha,\Omega}(y,w) + d_{\alpha,\Omega}(w,z) \\
        &\lesssim \sum_{k=1}^N 2^{-k} r_k^{\alpha-1} + 2^{-N} r_N^{\alpha-1} + \int_0^{r_N} t^{\alpha-1}\dx[t] \\
        &\lesssim N2^{-N}r_N^{\alpha-1} = N2^{-N} r_N^{-2/(p-1)}.
\end{align*}
Although $2^{-N}r_N^{\alpha-1}$ tends to $\infty$ as $N\to\infty$, the previous estimate allows to bound the values of $u$ on $S_N$ and hence their growth towards $S$. 
In fact, by~\eqref{eq:buck-stan} we obtain
\begin{align*}
    \abs{u(x)-u(z)} &\le C (d_{\alpha,\Omega}(x,z) + \sqrt{6})^{(p-1)/p} \\
        &\lesssim N 2^{-\frac{p-1}{p}N} r_N^{-2/p},
\end{align*}
where we can omit the constant additive term since the estimate for $d_{\alpha,\Omega}(x,z)$ dominates.
By truncating $u$ appropriately, we may without loss of generality assume that $u=0$ on $\Omega_1$ and therefore $u(x)=0$. 
Hence
\[
    \abs{u(z)}^q \lesssim N^q 2^{-\frac{(p-1)q}{p}N} r_N^{-2q/p}
\]
for all $z\in S_N$, where the constant is uniform in $z$ and independent of $N$.
We obtain
\begin{align*}
    \int_{\partial\Omega\setminus S} \abs{u(z)}^q\dx[\Hm[2](z)]
      &\lesssim \sum_{k=1}^\infty \Hm[2](S_k) k^q 2^{-\frac{(p-1)q}{p}k} r_k^{-2q/p} \\
      &\lesssim \sum_{k=1}^\infty k^q 2^{\bigl(1-\frac{(p-1)q}{p}\bigr)k} r_k^{1-\frac{2q}{p}}.
\end{align*}
So the approximative trace of $u$ is in $\Lp[q][\partial\Omega\setminus S]$ if
\[
    (2c-p+1)q < (c-1)p.
\]
For $q=p$ this is satisfied if $c+2<p$. So if we choose $c=5$ and $p=9$, then the approximative trace is not unique in $\Wone$, but every $u\in\Wone$ has an approximative trace in $\Lp[p][\partial\Omega]$. In fact, for $c=5$ and $p=9$ the above shows that
the approximative traces restricted to $\partial\Omega\setminus S$ are in $\Lp[q][\partial\Omega\setminus S]$ for all $q\in[1,18)$.
\end{example}

\begin{remark}
While Example~\ref{ex:p-conj-ate} does not give a counterexample to the exact question raised in the last paragraph of~\cite[Section~4]{AtE2011:DtN}, i.e.\ in the setting $p=q=2$, it indicates that an affirmative answer to this question is not to be expected.
\end{remark}

We briefly compare the approximative trace to Maz$'$ya and Burago's rough trace for elements of $\BV(\Omega)$.
For the required details we refer to~\cite[Section~9.5.1, Theorem~9.5.4 and Theorem~9.6.2]{Maz2011}.
A closely related notion is considered in~\cite[Section~5.10 and Theorem~5.10.7]{Ziemer1989}.
More recently, the rough trace was studied for more general domains with rectifiable boundary in~\cite{BK2010}, and an extension of Theorem~\ref{thm:mazbur} below can be found in~\cite[Theorems~2 and~6]{BK2010}.

\begin{definition}
Let $\Omega\subset\RR^d$ be open. For $u\in\BV(\Omega)$ and $z\in\mbdy\Omega$ define the \emphdef{rough trace} of $u$ in $z$ by
\[
    u^*(z) := \sup\{t\in\RR : z\in\mbdy[u>t]\}.
\]
\end{definition}
In~\cite[Section~9.5.1]{Maz2011} the rough trace is defined slightly differently with respect to the reduced boundary instead of the measure theoretic boundary. This is inconsequential for the following theorem.
It is possible to characterise the boundedness in $\Lone[\partial\Omega]$ of the rough trace under the assumption that $\Hm(\partial\Omega)<\infty$ and $\Hm(\partial\Omega\setminus\mbdy\Omega)=0$.
Note, however, that the latter assumption already implies uniqueness of the approximative trace for all $p$ by Proposition~\ref{prop:trace-unique-rect}. 

\begin{theorem}\label{thm:mazbur}
Let $\Omega\subset\RR^d$ be open with $\Hm(\partial\Omega)<\infty$ and $\Hm(\partial\Omega\setminus\mbdy\Omega)=0$. Then there exists a $C>0$ such that
\begin{equation}\label{eq:trace-int}
    \norm{u^*}_{L^1(\partial\Omega)} \le C\norm{u}_{\BV(\Omega)}
\end{equation}
for all $u\in \BV(\Omega)$ if and only if
there exists a $\delta>0$ and $M>0$ such that for all $A\subset\Omega$ with $\diam A<\delta$ one has 
\begin{equation}\label{eq:E-bdy-est}
    \Hm(\mbdy A\cap\partial\Omega)\le M\Hm(\mbdy A\cap\Omega).
\end{equation}
Moreover, if these equivalent conditions hold, then
\[
    u^*(z) = \lim_{r\to 0+} \frac{1}{\meas{B(z,r)\cap\Omega}} \int_{B(z,r)\cap\Omega} u(y)\dx[y]
\]
for $\Hm$-a.e.\ $z\in\partial\Omega$.
\end{theorem}

So if $\Hm(\partial\Omega)<\infty$, $\Hm(\partial\Omega\setminus\mbdy\Omega)=0$ and~\eqref{eq:E-bdy-est} holds for $\Omega$, then every element of $\smash{\Wt[1]}$ has a unique approximative trace in $L^1(\partial\Omega)$ by~\eqref{eq:trace-int} as one has $\norm{u}_{\BV(\Omega)}=\norm{u}_{1,1}$ and $\restrict{u}{\partial\Omega}=u^*$ for all $\smash{u\in\Wone[1]\cap\Cbar}$.
In~\cite[Theorem~1.3]{AtE2011:DtN} there is a version of~\eqref{eq:trace-int} for the case $p=2$.

We finally point out two applications for the results obtained in this paper.
In~\cite[Section~4, Robin boundary conditions]{DD2009} it is of interest when the approximative trace with values in $L^p(\Gamma)$ in $\Wone$ is unique, in which case they call the domain $\Omega$ admissible.
The present paper provides geometric criteria for when this is the case.

Moreover, the results and techniques developed in this paper will be helpful for the program
suggested in~\cite[Section~6]{BD2010:alt-fk} to extend their proof of the Faber--Krahn inequality for the Robin Laplacian to general domains, where the Robin Laplacian on general domains is defined as in~\cite{Dan2000:robin-bvp} and~\cite{AW03}.
While the validity of the Faber--Krahn inequality has been established more generally in~\cite{BG2010}, we strongly expect that the approximative trace considered here is the appropriate notion for the setting of~\cite{BD2010:alt-fk}.

\subsection*{Acknowledgements}
The author would like to thank Tom ter~Elst and Wolfgang Arendt for their support and interest in this project.


\begin{thebibliography}{ACMM01}

\bibitem[AH96]{AH96}
D.R. Adams and L.I. Hedberg, \emph{Function spaces and potential theory},
  Grundlehren der Mathematischen Wissenschaften, no. 314, Springer, Berlin,
  1996.

\bibitem[ACMM01]{ACMM2001}
L.~Ambrosio, V.~Caselles, S.~Masnou, and J.M. Morel, \emph{Connected components
  of sets of finite perimeter and applications to image processing}, J. Eur.
  Math. Soc. (JEMS) \textbf{3} (2001), 39--92.

\bibitem[AFP00]{AFP00}
L.~Ambrosio, N.~Fusco, and D.~Pallara, \emph{Functions of bounded variation and
  free discontinuity problems}, Oxford Mathematical Monographs, Oxford
  University Press, New York, 2000.

\bibitem[AD08]{AD08}
W.~Arendt and D.~Daners, \emph{Varying domains: stability of the {D}irichlet
  and the {P}oisson problem}, Discrete Contin. Dyn. Syst. \textbf{21} (2008),
  21--39.

\bibitem[AtE11]{AtE2011:DtN}
W.~Arendt and A.F.M. ter Elst, \emph{The {D}irichlet-to-{N}eumann operator on
  rough domains}, J. Differential Equations \textbf{251} (2011), 2100--2124.

\bibitem[AtE12]{AtE12:sect-form}
W.~Arendt and A.F.M. ter Elst, \emph{Sectorial forms and degenerate
  differential operators}, J. Operator Theory \textbf{67} (2012), 33--72.

\bibitem[AW03]{AW03}
W.~Arendt and M.~Warma, \emph{The {L}aplacian with {R}obin boundary conditions
  on arbitrary domains}, Potential Anal. \textbf{19} (2003), 341--363.

\bibitem[Bie05]{Bie2005:thesis}
M.~Biegert, \emph{Elliptic problems on varying domains}, {PhD} dissertation,
  Universit\"{a}t Ulm, 2005.

\bibitem[Bie09a]{Bie09:cap-mod}
M.~Biegert, \emph{On a capacity for modular spaces}, J. Math. Anal. Appl.
  \textbf{358} (2009), 294--306.

\bibitem[Bie09b]{Bie09:relcap-usem}
M.~Biegert, \emph{The relative capacity}, Ulmer Seminare \textbf{14} (2009),
  25--41.

\bibitem[BS01]{BS2001}
S.M. Buckley and A.~Stanoyevitch, \emph{Weak slice conditions and {H}\"older
  imbeddings}, J. London Math. Soc. (2) \textbf{64} (2001), 690--706.

\bibitem[BD10]{BD2010:alt-fk}
D.~Bucur and D.~Daners, \emph{An alternative approach to the {F}aber--{K}rahn
  inequality for {R}obin problems}, Calc. Var. Partial Differential Equations
  \textbf{37} (2010), 75--86.

\bibitem[BG10]{BG2010}
D.~Bucur and A.~Giacomini, \emph{A variational approach to the isoperimetric
  inequality for the {R}obin eigenvalue problem}, Arch. Ration. Mech. Anal.
  \textbf{198} (2010), 927--961.

\bibitem[BK10]{BK2010}
Y.D. Burago and N.N. Kosovski{\u\i}, \emph{The trace of {BV}-functions on an
  irregular subset}, Algebra i Analiz \textbf{22} (2010), 105--126, translation
  in St. Petersburg Math. J. 22 (2011), 251--266.

\bibitem[CHK16]{CHK2015}
R.~Chill, D.~Hauer, and J.~Kennedy, \emph{Nonlinear semigroups generated by
  j-elliptic functionals}, J. Math. Pures Appl. \textbf{105} (2016), 415--450.

\bibitem[Dan00]{Dan2000:robin-bvp}
D.~Daners, \emph{Robin boundary value problems on arbitrary domains}, Trans.
  Amer. Math. Soc. \textbf{352} (2000), 4207--4236.

\bibitem[DD09]{DD2009}
D.~Daners and P.~Dr{\'a}bek, \emph{A priori estimates for a class of
  quasi-linear elliptic equations}, Trans. Amer. Math. Soc. \textbf{361}
  (2009), 6475--6500.

\bibitem[DL08]{DeLe2008}
C.~De~Lellis, \emph{Rectifiable sets, densities and tangent measures}, Zurich
  Lectures in Advanced Mathematics, European Mathematical Society (EMS),
  Z\"urich, 2008.

\bibitem[EE87]{EE87}
D.E. Edmunds and W.D. Evans, \emph{Spectral theory and differential operators},
  Oxford Mathematical Monographs, Oxford University Press, New York, 1987.

\bibitem[EG92]{EG92}
L.C. Evans and R.F. Gariepy, \emph{Measure theory and fine properties of
  functions}, Studies in Advanced Mathematics, CRC Press, Boca Raton, FL, 1992.

\bibitem[Fal86]{Fal86}
K.J. Falconer, \emph{The geometry of fractal sets}, Cambridge Tracts in
  Mathematics, vol.~85, Cambridge University Press, Cambridge, 1986.

\bibitem[Fed45]{Fed45}
H.~Federer, \emph{The {G}auss-{G}reen theorem}, Trans. Amer. Math. Soc.
  \textbf{58} (1945), 44--76.

\bibitem[Fed46]{Fed46}
H.~Federer, \emph{Coincidence functions and their integrals}, Trans. Amer.
  Math. Soc. \textbf{59} (1946), 441--466.

\bibitem[Fed60]{Fed1960:np-surf}
H.~Federer, \emph{The area of a nonparametric surface}, Proc. Amer. Math. Soc.
  \textbf{11} (1960), 436--439.

\bibitem[Fed69]{Federer1969}
H.~Federer, \emph{Geometric measure theory}, Grundlehren der mathematischen
  Wissenschaften, no. 153, Springer, New York, 1969.

\bibitem[FZ73]{FZ72}
H.~Federer and W.P. Ziemer, \emph{The {L}ebesgue set of a function whose
  distribution derivatives are {$p$}-th power summable}, Indiana Univ. Math. J.
  \textbf{22} (1972/73), 139--158.

\bibitem[FF09]{FF09:fin-perim-plane}
A.~Ferriero and N.~Fusco, \emph{A note on the convex hull of sets of finite
  perimeter in the plane}, Discrete Contin. Dyn. Syst. Ser. B \textbf{11}
  (2009), 102--108.

\bibitem[Fre03]{Frem2003:vol2}
D.H. Fremlin, \emph{Measure theory. {V}olume 2: {B}road foundations}, Torres
  Fremlin, Colchester, 2003, Corrected second printing of the 2001 original.

\bibitem[Hed00]{Hed2000:stab}
L.I. Hedberg, \emph{Stability of {S}obolev spaces with zero boundary values},
  Function spaces and applications ({D}elhi, 1997), Narosa, New Delhi, 2000,
  pp.~91--97.

\bibitem[Kol81]{Kolsrud1981:example}
T.~Kolsrud, \emph{Approximation by smooth functions in {S}obolev spaces, a
  counterexample}, Bull. London Math. Soc. \textbf{13} (1981), 167--169.

\bibitem[Kri57]{Kri57}
K.~Krickeberg, \emph{Distributionen, {F}unktionen beschr\"ankter {V}ariation
  und {L}ebesguescher {I}nhalt nichtparametrischer {F}l\"achen}, Ann. Mat. Pura
  Appl. (4) \textbf{44} (1957), 92, 105--133.

\bibitem[Mat95]{Mat95}
P.~Mattila, \emph{Geometry of sets and measures in {E}uclidean spaces},
  Cambridge Studies in Advanced Mathematics, vol.~44, Cambridge University
  Press, Cambridge, 1995.

\bibitem[Maz11]{Maz2011}
V.~Maz'ya, \emph{Sobolev spaces with applications to elliptic partial
  differential equations}, augmented ed., Grundlehren der Mathematischen
  Wissenschaften, vol. 342, Springer, Heidelberg, 2011.

\bibitem[Mir64]{Mir64}
M.~Miranda, \emph{Superfici cartesiane generalizzate ed insiemi di perimetro
  localmente finito sui prodotti cartesiani}, Ann. Scuola Norm. Sup. Pisa (3)
  \textbf{18} (1964), 515--542.

\bibitem[Mor66]{Mor66}
C.B. Morrey, Jr., \emph{Multiple integrals in the calculus of variations},
  Grundlehren der mathematischen Wissenschaften, no. 130, Springer, New York,
  1966.

\bibitem[O'F97]{OFar97:sob-approx}
A.G. O'Farrell, \emph{An example on {S}obolev space approximation}, Bull.
  London Math. Soc. \textbf{29} (1997), 470--474.

\bibitem[Sau13]{MS2013:thesis}
M.~Sauter, \emph{Degenerate elliptic operators with boundary conditions via
  form methods}, {PhD} dissertation, The University of Auckland, Auckland (New
  Zealand), 2013.

\bibitem[Sch74]{Sch74}
H.H. Schaefer, \emph{{B}anach lattices and positive operators}, Grundlehren der
  mathematischen Wissenschaften, no. 215, Springer, New York, 1974.

\bibitem[Zie89]{Ziemer1989}
W.P. Ziemer, \emph{Weakly differentiable functions}, Graduate Texts in
  Mathematics, no. 120, Springer, New York, 1989.

\end{thebibliography}
\end{document}